\documentclass[a4paper,10pt]{amsart}
\usepackage[top=0.8in, bottom=0.8in, left=1.25in, right=1.25in]{geometry}
\usepackage{fancybox}
\usepackage{amsmath}
\usepackage{amssymb}
\usepackage{amsthm}
\usepackage{mathrsfs}
\usepackage{epstopdf}
\usepackage{epstopdf}
\usepackage{array}
\usepackage{amscd}
\usepackage{color}
\usepackage{hyperref}
\usepackage[colorinlistoftodos]{todonotes}
\usepackage[all]{xypic}

\usepackage{amsfonts}
\usepackage{amssymb}
\usepackage{mathrsfs}
\DeclareFontFamily{U}{rsfs}{%
\skewchar\font127}
\DeclareFontShape{U}{rsfs}{m}{n}{%
<-6>rsfs5<6-8.5>rsfs7<8.5->rsfs10}{}
\DeclareSymbolFont{rsfs}{U}{rsfs}{m}{n}
\DeclareSymbolFontAlphabet
{\mathrsfs}{rsfs}
\DeclareRobustCommand*\rsfs{%
\@fontswitch\relax\mathrsfs}

\theoremstyle{plain}
\newtheorem{thm}{Theorem}[section]
\newtheorem{prop}[thm]{Proposition}
\newtheorem{lem}[thm]{Lemma}

\newtheorem{cor}[thm]{Corollary}

\newtheorem{prop-defi}[thm]{Proposition-Definition}
\newtheorem{thm-defi}[thm]{Theorem-Definition}
\newtheorem{lem-defi}[thm]{Lemma-Definition}

\newtheorem{conj}[thm]{Conjecture}

\theoremstyle{definition}
\newtheorem{rmk}[thm]{Remark}
\newtheorem{defi}[thm]{Definition}

\newdimen\argwidth
\def\db[#1\db]{
 \setbox0=\hbox{$#1$}\argwidth=\wd0
 \setbox0=\hbox{$\left[\box0\right]$}
  \advance\argwidth by -\wd0
 \left[\kern.3\argwidth\box0 \kern.3\argwidth\right]}

\newcommand{\cC}{\mathcal{C}}
\newcommand\CC{\mathbb C}

\newcommand{\eE}{\mathcal{E}}
\newcommand{\fF}{\mathcal{F}}
\newcommand{\gG}{\mathcal{G}}
\newcommand{\hH}{\mathcal{H}}

\newcommand{\lL}{\mathcal{L}}

\newcommand{\oO}{\mathcal{O}}

\newcommand{\PP}{\mathbb{P}}
\newcommand{\QQ}{\mathbb{Q}}

\newcommand\RR{\mathbb R}

\newcommand{\sfV}{\mathsf{V}}

\newcommand{\zZ}{\mathcal{Z}}
\newcommand{\ZZ}{\mathbb{Z}}

\newcommand{\vir}{\mathrm{vir}}

\newcommand{\Hom}{\mathop{\rm Hom}\nolimits}

\newcommand{\dR}{\mathbf{R}}

\newcommand{\Hilb}{\mathop{\rm Hilb}\nolimits}

\newcommand{\Pic}{\mathop{\rm Pic}\nolimits}

\newcommand{\Chow}{\mathop{\rm Chow}\nolimits}

\newcommand{\ch}{\mathop{\rm ch}\nolimits}

\newcommand{\Ext}{\mathop{\rm Ext}\nolimits}
\newcommand{\Spec}{\mathop{\rm Spec}\nolimits}

\newcommand{\Coh}{\mathop{\rm Coh}\nolimits}

\newcommand{\cneq}{\mathrel{\raise.095ex\hbox{:}\mkern-4.2mu=}}
\newcommand{\eqcn}{\mathrel{=\mkern-4.5mu\raise.095ex\hbox{:}}}

\newcommand{\ext}{\mathop{\rm ext}\nolimits}

\newcommand{\HOM}{\mathop{{\hH}om}\nolimits}
\newcommand{\EXT}{\mathop{{\eE}xt}\nolimits}

\newcommand{\DT}{\mathop{\rm DT}\nolimits}
\newcommand{\PT}{\mathop{\rm PT}\nolimits}

\newcommand{\tr}{\mathop{\rm tr}\nolimits}

\newcommand{\RHom}{\mathop{\dR\mathrm{Hom}}\nolimits}
\newcommand\mdot{{\scriptscriptstyle\bullet}}

\makeatletter
 
  \@addtoreset{equation}{section}
\makeatother

\title[{Curve counting and DT/PT correspondence for CY 4-folds}]
{Curve counting and DT/PT correspondence  \\ for Calabi-Yau 4-folds}
\date{}
\author{Yalong Cao}
\address{Kavli Institute for the Physics and Mathematics of the Universe (WPI), The University of Tokyo Institutes for Advanced Study, The University of Tokyo, Kashiwa, Chiba 277--8583, Japan}
\email{yalong.cao@ipmu.jp}
\author{Martijn Kool}
\address{Mathematical Institute, Utrecht University, P.O.~Box 80010, 3508 TA Utrecht, The Netherlands}
\email{m.kool1@uu.nl}

\begin{document}
\maketitle
\begin{abstract}
Recently, Cao-Maulik-Toda defined stable pair invariants of a compact Calabi-Yau 4-fold $X$. 
Their invariants are conjecturally related to the Gopakumar-Vafa type invariants of $X$ defined using Gromov-Witten theory by Klemm-Pandharipande. In this paper, we consider curve counting invariants of $X$ using Hilbert schemes of curves and conjecture a DT/PT correspondence which relates these to stable pair invariants of $X$.

After providing evidence in the compact case, we define analogous invariants for toric Calabi-Yau 4-folds. 
We formulate a vertex formalism for both theories and conjecture a relation between the (fully equivariant) DT/PT vertex, which we check in several cases. This  relation implies a DT/PT correspondence for toric Calabi-Yau 4-folds with primary insertions.
\end{abstract}

\section{Introduction}

\subsection{GW/GV invariants}

Gromov-Witten invariants of a complex smooth projective variety $X$ are rational numbers (virtually) counting stable maps from curves to $X$. 
Let $X$ be a Calabi-Yau 4-fold\,\footnote{In this paper, a Calabi-Yau 4-fold is a complex smooth projective 4-fold $X$ satisfying $K_X \cong \oO_X$.}, then the virtual dimension formula shows that these invariants vanish unless $g=0$ or $1$. 
Since Gromov-Witten invariants are rational numbers, their enumerative meaning is a priori unclear.
In \cite{KP}, Klemm and Pandharipande defined Gopakumar-Vafa type invariants of $X$, in terms of Gromov-Witten invariants of $X$,
and conjectured their integrality.

More precisely, let $\overline{M}_{g,m}(X,\beta)$ be the moduli space of stable maps $f : C \rightarrow X$, where $C$ is a (connected) nodal curve of arithmetic genus $g$ with $m$ marked points satisfying $f_*[C] = \beta \in H_2(X)$.
This space has a virtual class with virtual dimension $1-g+m$. 
Consider the following Gromov-Witten invariants (with primary insertions)
\begin{align*}
\mathrm{GW}_{0, \beta}(X)(\gamma_1, \ldots, \gamma_m)
&:=\int_{[\overline{M}_{0, m}(X, \beta)]^{\rm{vir}}} \prod_{i=1}^m \mathrm{ev}_i^{\ast}(\gamma_i) \in \QQ, \\
\mathrm{GW}_{1, \beta}(X) &:= \int_{[\overline{M}_{1,0}(X, \beta)]^{\rm{vir}}}
1 \in \QQ,
\nonumber \end{align*}
where $\gamma_1, \ldots, \gamma_m \in H^*(X,\ZZ)$ and $\mathrm{ev}_i \colon \overline{M}_{0, m}(X, \beta)\to X$ is the evaluation map at the $i$th marked point. 
Klemm-Pandharipande defined genus 0 and 1 Gopakumar-Vafa type invariants 
\begin{align*}n_{0, \beta}(X)(\gamma_1, \ldots, \gamma_m), \quad n_{1,\beta}(X) \end{align*} 
by the following two ``multiple cover formulae''
\begin{align} 
\begin{split} \label{intro:n01}
& \sum_{\beta>0}\mathrm{GW}_{0, \beta}(X)(\gamma_1, \ldots, \gamma_m) \, q^{\beta}=\,
\sum_{\beta>0}n_{0, \beta}(X)(\gamma_1, \ldots, \gamma_m) \sum_{d=1}^{\infty}
d^{m-3}q^{d\beta}, \\
&\sum_{\beta>0}\mathrm{GW}_{1, \beta}(X)\,q^{\beta}=\,
\sum_{\beta>0} n_{1, \beta}(X) \sum_{d=1}^{\infty}
\frac{\sigma_1(d)}{d}q^{d\beta} \\
&\quad \quad \quad +\frac{1}{24}\sum_{\beta>0} n_{0, \beta}(X)(c_2(X))\log(1-q^{\beta}) -\frac{1}{24}\sum_{\beta_1, \beta_2}m_{\beta_1, \beta_2}
\log(1-q^{\beta_1+\beta_2}).
\end{split}
\end{align}
Here the sums are over all non-zero effective curve classes in $H_2(X)$ and $\sigma_1(d)=\sum_{i|d}i$. 
Moreover, $m_{\beta_1, \beta_2}\in\mathbb{Z}$ are so-called meeting invariants which are inductively determined by the genus 0 Gromov-Witten invariants of $X$.
Klemm-Pandharipande's integrality conjecture states that $n_{0, \beta}(X)(\gamma_1, \ldots, \gamma_m)$ and $n_{1,\beta}(X)$ are integers, which they verify in examples using virtual localization or mirror symmetry \cite{KP}.

\subsection{PT/GV correspondence} \label{PT/GV section}

Donaldson-Thomas theory of Calabi-Yau 4-folds was defined by Borisov-Joyce \cite{BJ} in general and Cao-Leung \cite{CL} in special cases.
In \cite{CMT2}, Cao-Maulik-Toda defined stable pair invariants of Calabi-Yau 4-folds, using the methods for constructing virtual classes of \cite{BJ}, and 
proposed an interpretation of \eqref{intro:n01} in terms of these invariants.

More precisely, denote by $P_n(X,\beta)$ the moduli space of stable pairs $\{ s:\oO_X\to F \}$ in the sense of Pandharipande-Thomas \cite{PT1}, where $F$ is a pure dimension 1 sheaf on $X$, $s$ is a section with 0-dimensional cokernel, and $\ch(F)=(0,0,0,\beta,n)$. 
This moduli space carries a virtual class  
\begin{equation} \label{vd}[P_n(X, \beta)]^{\rm{vir}}\in H_{2n}\big(P_n(X, \beta),\mathbb{Z}\big), \end{equation}
which depends on the choice of orientation of certain (real) line bundle on $P_n(X, \beta)$. On each connected component
of $P_n(X, \beta)$, there are two choices of orientations, which affects the corresponding contribution to the virtual class by a sign. Define 
\begin{align}\label{insertion for pt}
\tau \colon H^{*}(X,\ZZ)\to H^{*-2}(P_n(X,\beta),\ZZ), \quad
\tau(\gamma)=\pi_{P\ast}(\pi_X^{\ast}\gamma \cup\ch_3(\mathbb{F}) ),
\end{align}
where $\pi_X$ and $\pi_P$ are projections from $X \times P_n(X,\beta)$ to its factors and $\mathbb{I}^{\mdot}=\{\pi_X^*\oO_X\to \mathbb{F}\}$ is the universal stable pair on $X \times P_n(X,\beta)$. Since $\mathbb{F}$ is of pure dimension 1 on each fibre over $P_n(X,\beta)$, $\ch_3(\mathbb{F})$ is Poincar\'e dual to the class of the scheme-theoretic support $\cC$ of $\mathbb{F}$, which is a Cohen-Macaulay curve on each fibre over $P_n(X,\beta)$. 

Stable pair invariants (with primary insertions) of $X$ are defined by
\begin{align*}
P_{n,\beta}(X)(\gamma_1,\ldots,\gamma_m):=&\,\int_{[P_n(X,\beta)]^{\rm{vir}}} \prod_{i=1}^{m} \tau(\gamma_i) \in \ZZ, \\
P_{0,\beta}(X):=&\,\int_{[P_{0}(X,\beta)]^{\rm{vir}}}1 \in \mathbb{Z},   
\end{align*}
for all $\gamma_1, \ldots, \gamma_m \in H^{*}(X, \mathbb{Z})$. 
In the following conjectures, when the curve class $\beta$ is zero, we set 
$P_{0,0}(X):=1$, $P_{n,0}(X)(\gamma_1,\ldots,\gamma_m)=n_{0,0}(X)(\gamma_1,\ldots,\gamma_m):=0$ when $n>0$.

\begin{conj}\emph{(Conjecture 0.1 of \cite{CMT2} and Section 0.7 of \cite{CMT2})}\label{pair/GW g=0 conj intro}
Let $X$ be a Calabi-Yau 4-fold, $\beta \in H_2(X)$, $\gamma_1=\cdots=\gamma_n=\gamma \in H^4(X,\ZZ)$, and $n \in \ZZ_{>0}$.
Then there exists a choice of orientation such that
\begin{align*}
P_{n,\beta}(X)(\gamma_1,\ldots,\gamma_n)=\sum_{\begin{subarray}{c}\beta_0+\beta_1+\cdots+\beta_n=\beta  \\ \beta_0,\beta_1,\cdots,\beta_n\geqslant0 \end{subarray} }P_{0,\beta_0}(X)\cdot \prod_{i=1}^n n_{0,\beta_i}(X)(\gamma). \end{align*}
\end{conj}
\begin{conj}\emph{(Conjecture 0.2 of \cite{CMT2})}\label{pair/GW g=1 conj intro}
Let $X$ be a Calabi-Yau 4-fold and $\beta \in H_2(X)$. Then there exists a choice of orientation such that
\begin{align*}
\sum_{\beta \geqslant 0}
 P_{0,\beta}(X) \, q^{\beta}=
\prod_{\beta>0} M\big(q ^{\beta}\big)^{n_{1, \beta}(X)},
\end{align*}
where $M(q)=\prod_{n=1}^{\infty} (1-q^{n})^{-n}$ denotes the MacMahon function.
\end{conj}  
Conjecture \ref{pair/GW g=0 conj intro} can be interpreted as a wall-crossing formula in the derived category of the Calabi-Yau 4-fold \cite{CT1}. 
In \cite{CMT2, CKM}, several examples are computed to support both conjectures.

\subsection{DT/PT correspondence}

In this paper, we study curve counting invariants of $X$ using Hilbert schemes and conjecture an explicit DT/PT correspondence. As a consequence, we get an interpretation of the Gopakumar-Vafa type invariants $n_{0,\beta}(X)(\gamma_1,\ldots,\gamma_m)$ and $n_{1,\beta}(X)$ in terms of virtual counts of ideal sheaves of curves (closed 1-dimensional subschemes) on $X$. 

Let $I_n(X,\beta)$ denote the Hilbert scheme of closed subschemes $C \subseteq X$ of dimension $\leqslant 1$ 
such that $\ch(\oO_C)=(0,0,0,\beta,n)$. 
Similar to \cite{CMT2}, 
there exists a virtual class
\begin{equation*}
[I_n(X, \beta)]^{\rm{vir}}\in H_{2n}\big(I_n(X, \beta),\mathbb{Z}\big), 
\end{equation*}
in the sense of Borisov-Joyce \cite{BJ},
which depends on the choice of orientation of certain (real) line bundle on $I_n(X, \beta)$. When $\beta \neq 0$, as in the stable pairs case, we consider primary insertions
\begin{align*}
\tau \colon H^{*}(X,\ZZ)\to H^{*-2}(I_n(X,\beta),\ZZ), \quad
\tau(\gamma)=\pi_{I\ast}(\pi_X^{\ast}\gamma \cup\ch_3(\oO_{\zZ}) ),
\end{align*}
where $\pi_X$, $\pi_I$ are projections from $X \times I_n(X,\beta)$ to its factors, $\mathcal{Z}\subseteq X\times I_n(X,\beta)$ is the universal subscheme
and $\ch_3(\oO_{\mathcal{Z}})$ is Poincar\'e dual to the class of the maximal Cohen-Macaulay subscheme of $\mathcal{Z}$. We define curve counting invariants (with primary insertions) of $X$ by
\begin{align*}
I_{n,\beta}(X)(\gamma_1,\ldots,\gamma_m) :=&\, \int_{[I_n(X,\beta)]^{\rm{vir}}} \prod_{i=1}^{m} \tau(\gamma_i) \in \ZZ, \\
I_{0,\beta}(X):=&\, \int_{[I_{0}(X,\beta)]^{\rm{vir}}}1 \in \mathbb{Z}.  
\end{align*}

On Calabi-Yau 3-folds, curve counting and stable pair invariants are related by the DT/PT correspondence formulated in \cite{PT1} and proved in \cite{Bridgeland, Toda1} using wall-crossing and Hall algebra techniques. 
We conjecture an analogue on Calabi-Yau 4-folds.
\begin{conj}\emph{(Conjecture \ref{conj DT/PT})}\label{intro: conj DT/PT} 
Let $X$ be a smooth projective Calabi-Yau 4-fold, $\beta \in H_2(X)$, $\gamma_1, \ldots, \gamma_m \in H^*(X,\ZZ)$, and $n \in \ZZ$. 
Then there exists a choice of orientation such that
\begin{align*}(1) \,\, I_{0, \beta}(X)=P_{0, \beta}(X), \quad  (2) \,\, I_{n,\beta}(X)(\gamma_1,\ldots,\gamma_m)=P_{n,\beta}(X)(\gamma_1,\ldots,\gamma_m). \end{align*}
\end{conj}
In the Calabi-Yau 3-folds case, one requires corrections from ``floating points'' in the DT/PT correspondence. This is not needed in the case of Calabi-Yau 4-folds because the (``complex") virtual dimension of $I_n(X,\beta)$ and $P_n(X,\beta)$ is $n$ and primary insertions
are insensitive to floating points.
A detailed heuristic derivation for ``ideal Calabi-Yau 4-folds'', on which all curves deform in expected dimensions, is presented in Section \ref{heuristic argument}.  

The bulk of evidence in this paper comes from toric Calabi-Yau 4-folds discussed in the next section. For compact Calabi-Yau 4-folds we can check a few cases, where the moduli spaces $P_n(X,\beta)$ and $I_n(X,\beta)$ coincide (see also \cite{CMT1, CMT2}): 
\begin{itemize}
\item Suppose $\beta\in H_2(X)$ has the property that any element $C$ of the Chow group $\Chow_{\beta}(X)$ satisfies $\chi(\oO_C)=1$. Then Conjecture \ref{intro: conj DT/PT} holds for $n=0,1$. Examples include (1) irreducible curve classes on complete intersection Calabi-Yau 4-folds in products of projective spaces, (2) degree two classes on generic sextic 4-folds.
\item Suppose $X \rightarrow \PP^3$ is a Weierstrass elliptically fibred Calabi-Yau 4-fold and $\beta = r [F]$ is a multiple fibre class for some $r > 0$. Then Conjecture \ref{intro: conj DT/PT} holds for $n=0$.
\item Suppose $X = Y \times E$, where $Y$ is a smooth projective Calabi-Yau 3-fold, $E$ is an elliptic curve, and $\beta = r [E]$, where $[E]$ denotes the class of a  fibre $\{\mathrm{pt}\} \times E$ and $r>0$. Then Conjecture \ref{intro: conj DT/PT} holds for $n=0$.
\end{itemize}

\subsection{Equivariant DT/PT correspondence}

Let $X$ be a toric Calabi-Yau 4-fold\,\footnote{By this we mean a smooth quasi-projective toric 4-fold $X$ satisfying $K_X \cong \oO_X$ and $H^{>0}(\oO_X) = 0$.}. Denote by $(\CC^*)^4$ the dense open torus of $X$. Let $T \subseteq (\CC^*)^4$ be the 3-dimensional subtorus which preserves the Calabi-Yau volume form. Denote by $I_n(X,\beta)$ the Hilbert scheme of closed subschemes $Z \subseteq X$ with proper support, $[Z] = \beta$, and $\chi(\mathcal{O}_Z) = n$. Since $X$ is not compact, the Hilbert schemes $I_n(X,\beta)$ are in general non-compact as well, because ``floating points can move off to infinity''.  Nevertheless, the fixed locus
$$
I_n(X,\beta)^{T}=I_n(X,\beta)^{(\mathbb{C}^*)^4}
$$ 
consist of finitely many isolated reduced points (Lemma \ref{Tfixlocus} and Lemma \ref{Tfixlocus scheme}). 

For $\gamma_1, \ldots, \gamma_m \in H^*_T(X,\QQ)$, we define $T$-equivariant curve counting invariants of $X$ 
$$
I_{n,\beta}(X;T)(\gamma_1, \ldots, \gamma_m)\in \frac{\QQ(\lambda_1,\lambda_2,\lambda_3,\lambda_4)}{\langle \lambda_1+\lambda_2+\lambda_3+\lambda_4\rangle} \cong \QQ(\lambda_1,\lambda_2,\lambda_3)
$$
by using a localization formula (see Definition \ref{def of equi curve counting inv}). Here $\lambda_1, \ldots, \lambda_4$ are the equivariant parameters of $(\CC^*)^{4}$, which satisfy the relation $\lambda_1+\lambda_2+\lambda_3+\lambda_4 = 0$ on the Calabi-Yau torus $T$. This definition involves a sum over the elements of the fixed locus and depends on the choice of a sign at each torus fixed point. 
Therefore, we have $2^{\#I_n(X,\beta)^{T}}$ choices of signs in total. 
When $\beta=0$, $I_n(X,0)$ is the Hilbert scheme of $n$ points on $X$, 
which was studied by Nekrasov (with $K$-theoretic insertions \cite{Nekrasov}) and the authors (with tautological insertions \cite{CK}). 

Denote by $P_n(X,\beta)$ the moduli space of stable pairs $\{s : \mathcal{O}_X \rightarrow F\}$, where $F$ has proper support in class $\beta$ and $\chi(F) = n$. Contrary to $I_n(X,\beta)$, the moduli space $P_n(X,\beta)$ may not have 0-dimensional fixed locus. 
Nevertheless, there are many interesting cases for which $P_n(X,\beta)^{(\CC^*)^4}$ is 0-dimensional, for instance when $X$ is a local toric curve or surface. Then (by Proposition \ref{PTfixedlocus})
$$
P_n(X,\beta)^{T} = P_n(X,\beta)^{(\CC^*)^4},
$$
which consists of finitely many isolated reduced points. For $\gamma_1, \ldots, \gamma_m \in H^*_T(X)$, we may define $T$-equivariant stable pair invariants of $X$ by a localization formula as well 
(Definition \ref{def of equi pair inv})
$$
P_{n,\beta}(X;T)(\gamma_1, \ldots, \gamma_m)\in 
\QQ(\lambda_1,\lambda_2,\lambda_3).
$$
We form generating series
\begin{align*}
I_{\beta}(X;T)(\gamma_1, \ldots, \gamma_m):=&\, \sum_{n \in \ZZ} I_{n,\beta}(X;T)(\gamma_1, \ldots, \gamma_m)\, \, q^n  \in \QQ(\lambda_1,\lambda_2,\lambda_3)(\!(q)\!), \\
P_{\beta}(X;T)(\gamma_1, \ldots, \gamma_m):=&\, \sum_{n \in \ZZ} P_{n,\beta}(X;T)(\gamma_1, \ldots, \gamma_m)\, \, q^n  \in \QQ(\lambda_1,\lambda_2,\lambda_3)(\!(q)\!). \end{align*}
We denote the generating series without insertions by $I_{\beta}(X;T)$, $P_{\beta}(X;T)$ and set $I_{0,0}(X;T)=1$.

In Section \ref{vertex section}, we develop a 4-fold analogue of the vertex formalism introduced for toric 3-folds by Maulik-Nekrasov-Okounkov-Pandharipande \cite{MNOP}, Pandharipande-Thomas \cite{PT2}, and originating from physics \cite{AKMV}.
When $\beta=0$, this recovers the vertex formalism developed earlier by the authors in \cite{CK} and Nekrasov-Piazzalunga in \cite{NP}. As a result, we obtain the (fully equivariant) 4-fold DT and PT vertex
$$
\sfV_{\lambda\mu\nu\rho}^{\DT}(q), \, \sfV_{\lambda\mu\nu\rho}^{\PT}(q)  
\in \QQ(\lambda_1,\lambda_2,\lambda_3)(\!(q)\!),
$$
where $\lambda, \mu, \nu, \rho$ are finite plane partitions (i.e.~3-dimensional piles of boxes) and we only define the PT vertex when at most two of $\lambda, \mu, \nu, \rho$ are non-empty. Roughly speaking, these are the DT/PT partition functions on $X = \CC^4$ for which the underlying 
Cohen-Macaulay support curve is fixed and determined by plane partitions $\lambda, \mu, \nu, \rho$ along the coordinate axes of $\CC^4$. These generating functions depend on the choice of a sign for each $T$-fixed point.

\begin{rmk} \label{Nekrasovconj}
By a conjecture of Nekrasov (Conj.~\ref{nekrasov conj}), there exist choices of signs such that
\begin{align}\label{nek conj intro}
\mathsf{V}_{\varnothing\varnothing\varnothing\varnothing}^{\DT}(q) = \exp\Bigg(q \, \frac{(\lambda_1+\lambda_2)(\lambda_1+\lambda_3)(\lambda_2+\lambda_3)}{\lambda_1\lambda_2\lambda_3(\lambda_1+\lambda_2+\lambda_3)} \Bigg)\equiv \exp\bigg(-q\int_{\mathbb{C}^4} c_3^T(\mathbb{C}^4) \bigg),
\end{align}
where $\int_{\mathbb{C}^4}$ is the equivariant push-forward to a point defined by the Atiyah-Bott localization formula (see e.g. the proof of 
\cite[Theorem 3.13]{CK}).
In \cite{CK}, we conjectured that the choice of signs for which this equality holds is \emph{unique} and we provided computational evidence for this. 
\end{rmk}

The following is our DT/PT correspondence at the level of the 4-fold vertex:
\begin{conj}\emph{(Conjecture \ref{affine DT/PT conj})} \label{intro affine DT/PT conj}
For any finite plane partitions $\lambda, \mu, \nu, \rho$, at most two of which are non-empty, there are choices of signs such that
$$
\mathsf{V}_{\lambda\mu\nu\rho}^{\DT}(q) = \mathsf{V}_{\lambda\mu\nu\rho}^{\PT}(q) \, \mathsf{V}_{\varnothing\varnothing\varnothing\varnothing}^{\DT}(q).
$$
Suppose we choose the signs for $\mathsf{V}_{\varnothing\varnothing\varnothing\varnothing}^{\DT}(q)$ equal to the unique signs in Nekrasov's conjecture (Remark \ref{Nekrasovconj}). Then, at each order in $q$, the choice of signs for which LHS and RHS agree is unique up to an overall sign.
\end{conj}

\begin{rmk}
The formula of Remark \ref{Nekrasovconj} has a $K$-theoretic enhancement found by Nekrasov \cite{Nekrasov}. We discuss a $K$-theoretic enhancement of Conjecture \ref{intro affine DT/PT conj} in \cite{CKM1}.
\end{rmk}

The series $\mathsf{V}_{\lambda\mu\nu\rho}^{\DT}(q)$, $\mathsf{V}_{\lambda\mu\nu\rho}^{\PT}(q)$ are Laurent series. 
They have the same leading term $q^\ell$ for some $\ell \in \ZZ$ depending on $\lambda, \mu, \nu, \rho$. We define $\widetilde{\mathsf{V}}_{\lambda\mu\nu\rho}^{\DT}(q)= q^{-\ell} \mathsf{V}_{\lambda\mu\nu\rho}^{\DT}(q)$ and $\widetilde{\mathsf{V}}_{\lambda\mu\nu\rho}^{\PT}(q)= q^{- \ell}\mathsf{V}_{\lambda\mu\nu\rho}^{\PT}(q)$, which have leading term $q^0$ (though in general with coefficient $\neq 1$).
Using an implementation into Maple, we gathered the following evidence:
\begin{prop}\emph{(Proposition \ref{evidencethm})}\label{intro evidencethm} 
There are choices of signs such that 
$$\widetilde{\mathsf{V}}_{\lambda\mu\nu\rho}^{\DT}(q) = \widetilde{\mathsf{V}}_{\lambda\mu\nu\rho}^{\PT}(q) \, \mathsf{V}_{\varnothing\varnothing\varnothing\varnothing}^{\DT}(q) \mod q^N
$$
in the following cases:
\begin{itemize}
\item for any $|\lambda| + |\mu| + |\nu| + |\rho| \leqslant 1$ and $N=4$,
\item for any $|\lambda| + |\mu| + |\nu| + |\rho| \leqslant 2$ and $N=4$,
\item for any $|\lambda| + |\mu| + |\nu| + |\rho| \leqslant 3$ and $N=3$,
\item for any $|\lambda| + |\mu| + |\nu| + |\rho| \leqslant 4$ and $N=3$.
\end{itemize}
In each of these cases, the uniqueness statement of Conjecture \ref{intro affine DT/PT conj} holds.
\end{prop}
The last case involves a genuine asymptotic plane partition (i.e.~3D pile of boxes of height 2). \\

Conjecture \ref{intro affine DT/PT conj} implies the following (equivariant) DT/PT correspondence for toric Calabi-Yau 4-folds with primary insertions.
\begin{thm}\emph{(Theorem \ref{affine implies toric})}\label{intro affine implies toric}
Assume Conjecture \ref{intro affine DT/PT conj} holds. Let $X$ be a toric Calabi-Yau 4-fold, $\beta \in H_2(X)$ such that $P_n(X,\beta)^{(\CC^*)^4}$ is at most 0-dimensional for all $n \in \ZZ$, and let
$\gamma_1, \ldots, \gamma_m \in H^*_T(X)$. Then there exist choices of signs such that
\begin{align*}
\frac{I_{\beta}(X;T)(\gamma_1, \ldots, \gamma_m)}{I_{0}(X;T)}=P_{\beta}(X;T)(\gamma_1, \ldots, \gamma_m).
\end{align*}
In particular, 
without insertions we have
\begin{align*}
\frac{I_{\beta}(X;T)}{I_{0}(X;T)}= P_{\beta}(X;T).
\end{align*}
\end{thm}
Combining Proposition \ref{intro evidencethm} and Theorem \ref{intro affine implies toric}, we obtain verifications of the (equivariant) DT/PT correspondence for several local toric geometries (see Section \ref{examplessec} for details).

Although the Hilbert schemes $I_n(X,\beta)$ are rarely compact, there are many cases where the moduli space $P_n(X,\beta)$ is compact. Then the virtual class $[P_{n}(X,\beta)]^{\vir}$ of the previous section is defined. If $P_n(X,\beta)^{(\CC^*)^4}$ is also 0-dimensional,
the invariants of this section and the invariants of the previous section are expected\,\footnote{In Theorem \ref{check loc formula}, we prove the virtual localization formula in some special cases for local surfaces. The general case is currently still open. } to be related by a 4-fold virtual localization formula
\begin{equation} \label{virloc1}
P_{n,\beta}(X)(\gamma_1, \ldots, \gamma_m) = P_{n,\beta}(X;T)(\gamma_1, \ldots, \gamma_m) \in \ZZ
\end{equation}
for any $\gamma_1, \ldots, \gamma_m \in H^*(X,\ZZ)$ admitting $T$-equivariant lifts and satisfying $\sum_{i} \deg \tau(\gamma_i) = 2n$, and for appropriately chosen signs on the RHS. Similarly, for any $n<0$ we have $[P_{n}(X,\beta)]^{\vir} = 0$ (by \eqref{vd}) and therefore we expect that there exist choices of signs such that
\begin{equation} \label{virloc2}
P_{n,\beta}(X;T) = 0, \qquad \textrm{for all } n < 0.
\end{equation}
For the following toric Calabi-Yau 4-folds, $P_n(X,\beta)$ is compact and $P_n(X,\beta)^{(\CC^*)^4}$ is isolated for all $\beta \in H_2(X)$ and $n \in \ZZ$:
\begin{itemize}
\item $X = \mathrm{Tot}_{\PP^2}(\oO(-1) \oplus \oO(-2))$, i.e.~the total space of $\oO_{\PP^2}(-1) \oplus \oO_{\PP^2}(-2)$
\item $X =\mathrm{Tot}_{\PP^1 \times \PP^1}(\oO(-1,-1) \oplus \oO(-1,-1))$. 
\end{itemize}
Despite $X$ being non-compact in these cases, their PT moduli spaces are compact and the
following theorems could be seen as evidence for Conjecture \ref{intro: conj DT/PT}. The key point is that, by the
localization formulae (\ref{virloc1}) and (\ref{virloc2}), the only nonzero terms contributing in Theorem \ref{intro affine implies toric}
are those for which the degree of the insertion matches the degree of the virtual class.
\begin{thm} \emph{(Theorem \ref{mainthm1})} 
Let $X$ be a toric Calabi-Yau 4-fold and $\beta \in H_2(X)$. Suppose $P_n(X, \beta)$ is proper and $P_n(X, \beta)^{(\CC^*)^4}$ is at most 0-dimensional for all $n \in \ZZ$. Assume the following:
\begin{enumerate} 
\item the DT/PT vertex correspondence (Conjecture \ref{intro affine DT/PT conj}) holds,
\item \eqref{virloc1} holds for $\beta$ and $n=0$, and \eqref{virloc2} holds for $\beta$ and all $n < 0$, 
\item the signs of (1) and (2) can be chosen compatibly. 
\end{enumerate}
Then
$
I_{0,\beta}(X;T)=P_{0,\beta}(X) \in \ZZ.
$
\end{thm}
\begin{thm} \emph{(Theorem \ref{mainthm2})}
Let $X$ be a toric Calabi-Yau 4-fold satisfying $\int_X c_3^T(X)=0$ and let $\beta \in H_2(X)$. Let $\gamma_1, \ldots, \gamma_m \in H^*(X,\ZZ)$ admitting $T$-equivariant lifts and satisfying $\sum_i \deg \tau(\gamma_i) = 2n > 0$. Suppose $P_\chi(X, \beta)$ is proper and $P_\chi(X, \beta)^{(\CC^*)^4}$ is at most 0-dimensional for all $\chi \in \ZZ$. Assume the following:
\begin{enumerate} 
\item the DT/PT vertex correspondence (Conjecture \ref{intro affine DT/PT conj}) holds,
\item Nekrasov's conjecture (Remark \ref{Nekrasovconj}) holds,
\item \eqref{virloc1} holds for $\beta$, $\gamma_1, \ldots, \gamma_m$, $n$, 
\item the signs of (1), (2), and (3) can be chosen compatibly. 
\end{enumerate}
Then 
$
I_{n,\beta}(X;T)(\gamma_1, \ldots, \gamma_m)=P_{n,\beta}(X)(\gamma_1, \ldots, \gamma_m) \in \ZZ.
$
\end{thm}

\begin{rmk}
The spaces $P_n(X, \beta)$ are compact and the fixed loci $P_n(X, \beta)^{(\CC^*)^4}$ are at most 0-dimensional for all $n \in \ZZ$ when $X= \mathrm{Tot}_{\PP^2}(\oO(-1) \oplus \oO(-2))$ or $\mathrm{Tot}_{\PP^1 \times \PP^1}(\oO(-1,-1) \oplus \oO(-1,-1)$. In the second case, the equivariant integral $\int_X c_3^T(X)$ is zero by a direct computation. In Section \ref{examplessec}, we discuss these two local surfaces and the local curve $\mathrm{Tot}_{\PP^1}(\oO \oplus \oO(-1) \oplus \oO(-1))$ in more detail. In the range where we checked Conjecture \ref{intro affine DT/PT conj} (Proposition \ref{intro evidencethm}) and Nekrasov's formula (Remark \ref{Nekrasovconj}), 
the previous theorems are only conditional on the virtual localization formulae \eqref{virloc1}, \eqref{virloc2} (and compatibility of signs).
In an appendix, we also prove the virtual localization formula in some cases. 
\end{rmk}

\subsection{DT/PT generating series of a local curve}

Let $X=\mathrm{Tot}_{\PP^1}(\oO \oplus \oO(-1) \oplus \oO(-1))$. Then $H_2(X) \cong \ZZ$ is freely generated by the class of the zero section. For any $d \geqslant 1$ and $n \in \ZZ$, we consider (equivariant) stable pair invariant without insertions
\begin{align*}P_{n,d}(X; T)\in \QQ(\lambda_1,\lambda_2,\lambda_3). \end{align*}
Motivated by \cite{CMT2}, we conjecture the following expression for the generating series:
\begin{conj}\emph{(Conjecture \ref{pair inv gene series})}\label{intro pair inv gene series}
For $X=\mathrm{Tot}_{\PP^1}(\oO \oplus \oO(-1) \oplus \oO(-1))$, there exist choices of signs such that the following equation holds
\begin{align*}\sum_{n,d\geqslant0}P_{n,d}(X; T)\,q^n y^d=\exp\Big(\frac{q y}{\lambda_2}\Big), \end{align*}
where $\lambda_2$ is the equivariant parameter for the $\CC^*$-action on the first fibre $\oO_{\mathbb{P}^1}$ and $P_{0,0}(X; T):=1$.
\end{conj}
By using the vertex formalism and Maple calculations, we obtain the following verifications:
\begin{prop}\emph{(Proposition \ref{verify pair inv gene series})}
Conjecture \ref{intro pair inv gene series} holds in the following cases: 
\begin{itemize}
\item for any $n \leqslant d$, 
\item $d=1$ and modulo $q^6$,
\item $d=2$ and modulo $q^6$,
\item $d=3$ and modulo $q^6$,
\item $d=4$ and modulo $q^7$.
\end{itemize}
\end{prop}
As an application, we obtain a conjectural generating series for the curve counting invariants of $X=\mathrm{Tot}_{\PP^1}(\oO \oplus \oO(-1) \oplus \oO(-1))$.
Assume Nekrasov's formula (\ref{nek conj intro}), Conjectures \ref{intro affine DT/PT conj}, \ref{intro pair inv gene series}, 
and assume the signs can be chosen compatibly, then 
\begin{align*}\sum_{n,d\geqslant0}I_{n,d}(X; T)\,q^n y^d=
\exp\bigg(\frac{q}{\lambda_2}\bigg(y+\frac{(\lambda_1+\lambda_2)(\lambda_1+\lambda_3)(\lambda_2+\lambda_3)}{\lambda_1\lambda_3(\lambda_1+\lambda_2+\lambda_3)}+\frac{\lambda_3(\lambda_1-\lambda_2)(\lambda_1+\lambda_2+\lambda_3)}{\lambda_1(\lambda_1+\lambda_3)(\lambda_2+\lambda_3)}\bigg)\bigg). \end{align*}

\subsection{Notation and conventions}

In this paper, all varieties and schemes are defined over $\mathbb{C}$. 
For $\fF, \gG \in \mathrm{D^{b}(Coh(\textsl{X}))}$, we denote by $\mathrm{ext}^i(\fF, \gG)$ the dimension of 
$\Ext^i_X(\fF, \gG)$. 
A class $\beta\in H_2(X)$ is called irreducible (resp.~primitive) if it is not the sum of two non-zero effective classes
(resp. if it is not a positive integer multiple of an effective class).

\subsection{Acknowledgements}

Y. C. is grateful to Davesh Maulik and Yukinobu Toda for previous collaboration \cite{CMT2}, 
without which the current work would not exist. We are grateful to Sergej Monavari for helpful discussions and 
the anonymous referee for useful suggestions which improved the exposition of the paper.
Part of this work was done while Y. C. was in Oxford, supported by the Royal Society Newton International Fellowship.
Y. C. is partially supported by the World Premier International Research Center Initiative (WPI), MEXT, Japan, 
JSPS KAKENHI Grant Number JP19K23397 and Royal Society Newton International Fellowships Alumni 2019.

\section{$\mathrm{DT/PT}$ for compact Calabi-Yau 4-folds}

\subsection{DT invariants of Calabi-Yau 4-folds}

Let us first introduce the set-up of Donaldson-Thomas invariants of smooth projective Calabi-Yau 4-folds $X$.
We fix an ample divisor $\omega$ on $X$
and take a cohomology class
$v \in H^{\ast}(X, \mathbb{Q})$.

The coarse moduli space $M_{\omega}(v)$
of $\omega$-Gieseker semistable sheaves
$E$ on $X$ with $\ch(E)=v$ exists as a projective scheme.
We always assume that
$M_{\omega}(v)$ is a fine moduli space, i.e.~any point $[E] \in M_{\omega}(v)$ is stable and
there is a universal family
\begin{align*}
\eE \in \Coh(X \times M_{\omega}(v))
\end{align*}
flat over $M_\omega(v)$.
For instance, the moduli space of 1-dimensional stable sheaves $E$ with $[E]=\beta$, $\chi(E)=1$ and Hilbert schemes of 
closed subschemes satisfy this assumption \cite{Cao, CK, CMT1, CT2}.

In~\cite{BJ, CL}, under certain hypotheses,
the authors construct 
a virtual
class
\begin{align}\label{virtual}
[M_{\omega}(v)]^{\rm{vir}} \in H_{2-\chi(v, v)}(M_{\omega}(v), \mathbb{Z}), \end{align}
where $\chi(-,-)$ denotes the Euler pairing.
Notice that this class could a priori be non-algebraic.

Roughly speaking, in order to construct such a class, one chooses at
every point $[E]\in M_{\omega}(v)$, a half-dimensional real subspace
\begin{align*}\Ext_{+}^2(E, E)\subseteq \Ext^2(E, E)\end{align*}
of the usual obstruction space $\Ext^2(E, E)$, on which the quadratic form $Q$ defined by Serre duality is real and positive definite. 
Then one glues local Kuranishi-type models of the form 
\begin{equation}\kappa_{+}=\pi_+\circ\kappa: \Ext^{1}(E,E)\to \Ext_{+}^{2}(E,E),  \nonumber \end{equation}
where $\kappa$ is the Kuranishi map for $M_{\omega}(v)$ at $[E]$ and $\pi_+$ denotes projection 
on the first factor of the decomposition $\Ext^{2}(E,E)=\Ext_{+}^{2}(E,E)\oplus\sqrt{-1}\cdot\Ext_{+}^{2}(E,E)$.  

In \cite{CL}, local models are glued in three special cases: 
\begin{enumerate}
\item when $M_{\omega}(v)$ consists of locally free sheaves only, 
\item  when $M_{\omega}(v)$ is smooth,
\item when $M_{\omega}(v)$ is a shifted cotangent bundle of a quasi-smooth derived scheme. 
\end{enumerate}
In each case, the corresponding virtual classes are constructed using either gauge theory or algebro-geometric perfect obstruction theory.

The general gluing construction is due to Borisov-Joyce \cite{BJ}, 
based on Pantev-T\"{o}en-Vaqui\'{e}-Vezzosi's theory of shifted symplectic geometry \cite{PTVV} and Joyce's theory of derived $C^{\infty}$-geometry.
The corresponding virtual class is constructed using Joyce's
D-manifold theory (a machinery similar to Fukaya-Oh-Ohta-Ono's theory of Kuranishi space structures used for defining Lagrangian Floer theory).

To construct the above virtual class (\ref{virtual}) with coefficients in $\mathbb{Z}$ (instead of $\mathbb{Z}_2$), we need an orientability result 
for $M_{\omega}(v)$, which can be stated as follows.
Let  
\begin{equation*}
 \lL:=\mathrm{det}(\dR \hH om_{\pi_M}(\eE, \eE))
 \in \Pic(M_{\omega}(v)), \quad  
\pi_M \colon X \times M_{\omega}(v)\to M_{\omega}(v)
\end{equation*}
be the determinant line bundle of $M_{\omega}(v)$, which is equipped with the nondegenerate symmetric pairing $Q$ induced by Serre duality.  An \textit{orientation} of 
$(\mathcal{L},Q)$ is a reduction of its structure group from $\mathrm{O}(1,\mathbb{C})$ to $\mathrm{SO}(1, \mathbb{C})=\{1\}$. In other words, we require a choice of square root of the isomorphism
\begin{equation*}Q: \lL\otimes \lL \to \oO_{M_{\omega}(v)}.  \end{equation*}
Existence of orientations was first proved when the Calabi-Yau 4-fold $X$ satisfies 
$\mathrm{Hol}(X)=\mathrm{SU}(4)$ and $H^{\rm{odd}}(X,\mathbb{Z})=0$ \cite{CL2}, 
and was recently generalized to arbitrary Calabi-Yau 4-folds \cite{CGJ}.
Notice that the collection of orientations forms a torsor for $H^{0}(M_{\omega}(v),\mathbb{Z}_2)$. 

The computational examples in this section only involve the virtual class constructions in situations (2) and (3) mentioned above.
We briefly review them (modulo discussions on choices of orientations) as follows:  
\begin{itemize}
\item When $M_{\omega}(v)$ is smooth, the obstruction sheaf $\mathrm{Ob}\to M_{\omega}(v)$ is a vector bundle endowed with a quadratic form $Q$ via Serre duality. Then the virtual class is given by
\begin{equation}[M_{\omega}(v)]^{\rm{vir}}=\mathrm{PD}(e(\mathrm{Ob},Q)).   \nonumber \end{equation}
Here $e(\mathrm{Ob}, Q)$ is the half-Euler class of 
$(\mathrm{Ob},Q)$ (i.e.~the Euler class of its real form $\mathrm{Ob}_+$), 
and $\mathrm{PD}(-)$ denotes 
Poincar\'e dual. 
Note that the half-Euler class satisfies 
\begin{align*}
e(\mathrm{Ob},Q)^{2}&=(-1)^{\frac{\mathrm{rk}(\mathrm{Ob})}{2}}e(\mathrm{Ob}),  \qquad \textrm{ }\mathrm{if}\textrm{ } \mathrm{rk}(\mathrm{Ob})\textrm{ } \mathrm{is}\textrm{ } \mathrm{even}, \\
 e(\mathrm{Ob},Q)&=0, \qquad\qquad\qquad\qquad \ \,  \textrm{ }\mathrm{if}\textrm{ } \mathrm{rk}(\mathrm{Ob})\textrm{ } \mathrm{is}\textrm{ } \mathrm{odd}. 
\end{align*}
\item Suppose $M_{\omega}(v)$ is the classical truncation of the shifted cotangent bundle of a quasi-smooth derived scheme. Roughly speaking, this means that at any closed point $[E]\in M_{\omega}(v)$, we have a Kuranishi map of the form
\begin{equation}\kappa \colon
 \Ext^{1}(E,E)\to \Ext^{2}(E,E)=V_E\oplus V_E^{*},  \nonumber \end{equation}
where $\kappa$ factors through a maximal isotropic subspace $V_E$ of $(\Ext^{2}(E,E),Q)$. Then the virtual class of $M_{\omega}(v)$ is, 
roughly speaking, the 
virtual class of the perfect obstruction theory formed by $\{V_E\}_{E\in M_{\omega}(v)}$. 
When $M_{\omega}(v)$ is furthermore smooth as a scheme, 
then it is
simply the Euler class of the vector bundle 
$\{V_E\}_{E\in M_{\omega}(v)}$ over $M_{\omega}(v)$. 
\end{itemize}

\subsection{DT/PT correspondence} 

Let $X$ be a smooth projective Calabi-Yau 4-fold and $\beta \in H_2(X)$. Let $I_n(X,\beta)$ be the Hilbert scheme of closed subschemes $ Z \subseteq X$ of dimension $\leqslant 1$
with $\ch(\oO_Z)=(0,0,0,\beta,n)$. 
These Hilbert schemes are isomorphic to moduli spaces of rank 1 torsion free sheaves on $X$ with trivial determinant and $\ch_2=0$. On the latter, one can construct a virtual class 
\begin{equation*}[I_n(X, \beta)]^{\rm{vir}}\in H_{2n}\big(I_n(X, \beta)\big), \end{equation*}
in the sense of Borisov-Joyce \cite{BJ} (as in \cite[Theorem 1.4]{CMT2}).
This virtual class depends on the choice of orientation discussed above.

When $\beta \neq 0$, we consider primary insertions
\begin{align}\label{insert}
\tau \colon H^{*}(X,\ZZ)\to H^{*-2}(I_n(X,\beta),\ZZ),  \quad 
\tau(\gamma)=\pi_{I\ast}(\pi_X^{\ast}\gamma \cup\ch_3(\mathbb{\oO_{\mathcal{Z}}}) ),
\end{align}
where $\pi_X$, $\pi_I$ are projections from $X \times I_n(X,\beta)$
to the corresponding factors, $\mathcal{Z}\subseteq X\times I_n(X,\beta)$ is the universal subscheme and $\ch_3(\mathbb{\oO_{\mathcal{Z}}})$ is the Poincar\'e dual of the class of the maximal Cohen-Macaulay subscheme of $\mathcal{Z}$.

For any $\gamma_1, \ldots, \gamma_m \in H^*(X,\ZZ)$, we define curve counting invariants (with primary insertions) of $X$ as follows
\begin{align*}
I_{n,\beta}(X)(\gamma_1,\ldots,\gamma_m):=&\,\int_{[I_n(X,\beta)]^{\rm{vir}}} \prod_{i=1}^{m} \tau(\gamma_i) \in \ZZ, \\
I_{0,\beta}(X):=&\,\int_{[I_{0}(X,\beta)]^{\rm{vir}}}1 \in \mathbb{Z}. 
\end{align*}
We propose the following $\mathrm{DT/PT}$ correspondence on compact Calabi-Yau 4-folds:
\begin{conj}\label{conj DT/PT}
Let $X$ be a smooth projective Calabi-Yau 4-fold, $\beta \in H_2(X)$, $\gamma_1, \ldots, \gamma_m \in H^*(X,\ZZ)$, and $n \in \ZZ$. 
Then there exists a choice of orientation such that
\begin{align*}(1) \,\, I_{0, \beta}(X)=P_{0, \beta}(X), \quad  (2) \,\, I_{n,\beta}(X)(\gamma_1,\ldots,\gamma_m)=P_{n,\beta}(X)(\gamma_1,\ldots,\gamma_m). \end{align*}
\end{conj}

\subsection{Heuristic argument}\label{heuristic argument}

In this subsection, we give a heuristic argument to explain part of Conjecture \ref{conj DT/PT}. 
The argument is similar to the one used in \cite{CMT2}. 

Let $X$ be an ``ideal'' smooth projective Calabi-Yau 4-fold, in the sense that curves of $X$ deform in families of expected dimensions, and have expected generic properties as follows:  
\begin{enumerate}
\item Connected reduced curves with arithmetic genus 0 (rational curves). 
Any rational curve in $X$ is a chain of smooth $\mathbb{P}^1$'s with normal bundle $\mathcal{O}_{\mathbb{P}^{1}}(-1) \oplus \oO_{\PP^1}(-1) \oplus \oO_{\PP^1}$ and
moves in a compact 1-dimensional smooth family of embedded rational curves, whose general member is smooth with 
normal bundle $\mathcal{O}_{\mathbb{P}^{1}}(-1) \oplus \oO_{\PP^1}(-1) \oplus \oO_{\PP^1}$.
\item Connected reduced curves with arithmetic genus 1. Any such curve $E$ in $X$ is smooth (i.e.~an elliptic curve) and super-rigid, i.e.~its normal bundle is 
$L_1 \oplus L_2 \oplus L_3$
for general degree zero line bundle $L_i$ on $E$
satisfying $L_1 \otimes L_2 \otimes L_3=\oO_E$. 
Also, any two elliptic curves on $X$ are 
disjoint and disjoint from all families of rational curves on $X$.
\item
There are no connected, reduced curves in $X$ of arithmetic genus $g\geqslant 2$.
\end{enumerate}

${}$ \\
\textbf{$\bullet$ $I_{0, \beta}(X)=P_{0, \beta}(X)$}.
We claim $I_{0}(X,\beta)\cong P_{0}(X,\beta)$. For $[I_Z]\in I_{0}(X,\beta)$, the torsion subsheaf $T_0 \subseteq \oO_Z$ gives an exact sequence
\begin{align*}0\to T_0 \to \oO_Z \to \oO_C\to 0, \end{align*}
where $T_0$ is 0-dimensional and $C$ is Cohen-Macaulay.
From our ``ideal'' assumptions, a 1-dimensional Cohen-Macaulay scheme $C_0$ supported in one of our families of rational curves 
(resp.~elliptic curves) has $\chi(\oO_{C_0})\geqslant 1$ (resp. $\chi(\oO_{C_0})\geqslant 0$). Since $\chi(\oO_Z)=0$, $C$ can only be supported on some elliptic curves in $X$ and $T_0=0$.
Thus $[I_Z] = [I_C] \in P_{0}(X,\beta)$ defines a stable pair. 

Conversely, for $[I^{\mdot}=\{\oO_X\to F\}] \in P_{0}(X,\beta)$, we have a short exact sequence
\begin{align*}0\to \oO_C \to F \to Q\to 0, \end{align*}
where $C$ is Cohen-Macaulay and $Q$ is 0-dimensional. By a similar reasoning, we know $Q=0$, $C$ is supported on some elliptic curves in $X$, and $[I_C] \in I_{0}(X,\beta)$.
The virtual classes get identified under this isomorphism and hence $I_{0, \beta}(X)=P_{0, \beta}(X)$. 
 
${}$ \\
\textbf{$\bullet$ $I_{n,\beta}(X)(\gamma_1,\ldots,\gamma_m)=P_{n,\beta}(X)(\gamma_1,\ldots,\gamma_m)$}. 
Here we only justify the case when $m=n$ and $\gamma_1=\cdots=\gamma_m=\gamma\in H^4(X)$.
For $n\geqslant 1$, we want to compute 
\begin{align*}\int_{[I_{n}(X,\beta)]^{\rm{vir}}}\tau(\gamma)^n, \quad \gamma\in H^4(X,\mathbb{Z}),  \end{align*}
when $X$ is an ideal smooth projective Calabi-Yau 4-fold.
Let $\{Z_i\}_{i=1}^n$ be (mutually distinct) 4-cycles which represent the class $\gamma$. For dimension reasons,
we may assume that for any $i\neq j$
the rational curves meeting with $Z_i$ are
disjoint from those meeting $Z_j$, and the elliptic curves are disjoint from all $Z_i$. 
The insertions cut down the moduli space to finitely many elements whose support intersect all $\{Z_i\}_{i=1}^n$.
We denote the moduli space of such ``incident'' 1-dimensional subschemes by
\begin{align*}Q_{n}(X,\beta;\{Z_i\}_{i=1}^n)\subseteq I_{n}(X,\beta).  \end{align*}
Then we claim that 
\begin{align}\label{Q_n:identity}
Q_{n}(X,\beta;\{Z_i\}_{i=1}^n)=\coprod_{\begin{subarray}{c}
\beta_0+\beta_1+\cdots+\beta_n=\beta  
\end{subarray}}I_{0}(X,\beta_0)\times R_1(X,\beta_1;Z_1)\times \cdots \times R_1(X,\beta_n;Z_n), \end{align}
where $R_1(X,\beta_i;Z_i)$ is the moduli space of 1-dimensional subschemes 
supported on rational curves (in class $\beta_i$) which meet with $Z_i$. 

In fact, for $Z\in Q_{n}(X,\beta;\{Z_i\}_{i=1}^n)$, we have a torsion filtration of $\oO_Z$:
\begin{align*}0\to T_0 \to \oO_Z \to \oO_C\to 0, \end{align*}
where $T_0$ is 0-dimensional and $C$ is Cohen-Macaulay. Note that $\oO_C$ decomposes into a direct sum 
$\bigoplus_{i=0}^n \oO_{C_{i}}$, where 
$C_0$ is supported on elliptic curves and 
each $C_i$ for $1\leqslant i\leqslant n$ is supported on \emph{disjoint} rational curves which meet with $Z_i$. 
As explained before, a Cohen-Macaulay scheme $C'$ supported on a family of rational curves (resp. elliptic curves) satisfies $\chi(\oO_{C'})\geqslant 1$ (resp. $\chi(\oO_{C'})\geqslant 0$), so $\chi(\oO_{C_0})\geqslant 0$, $\chi(\oO_{C_i})\geqslant 1$ for all $1\leqslant i \leqslant n$. 
Since  $\chi(\oO_Z)=n$, we have $\chi(\oO_{C_0})=0$, $\chi(\oO_{C_i})=1$ for all $i=1, \ldots, n$, and $T_0=0$. Therefore (\ref{Q_n:identity}) holds. 

Moreover, each $R_1(X,\beta_i;Z_i)$ consists of 
finitely many rational curves which meet with $Z_i$, 
whose number is exactly $n_{0, \beta_i}(\gamma)$. 
By counting the number of points in $I_{0}(X,\beta_0)$ and 
$R_1(X,\beta_i;Z_i)$, for all $1 \leqslant i \leqslant n$, we obtain  
\begin{align*}
I_{n, \beta}(X)(\gamma,\ldots,\gamma) :=&\,\int_{[I_{n}(X,\beta)]^{\rm{vir}}}\tau(\gamma)^n=\int_{[Q_{n}(X,\beta;\{Z_i\}_{i=1}^n)]^{\mathrm{vir}}}1 \\
=&\,\sum_{\begin{subarray}{c}\beta_0+\beta_1+\cdots+\beta_n=\beta  \\ \beta_0,\beta_1,\cdots,\beta_n\geqslant0 \end{subarray} }I_{0,\beta_0}(X) \cdot \prod_{i=1}^n n_{0,\beta_i}(\gamma). \end{align*}
Combining with $I_{0, \beta}(X)=P_{0, \beta}(X)$ explained above, we see that
the above equality coincides with the formula in Conjecture \ref{pair/GW g=0 conj intro}, so 
$I_{n,\beta}(X)(\gamma,\ldots,\gamma)=P_{n,\beta}(X)(\gamma,\ldots,\gamma)$ for $n>0$.

\subsection{Compact examples}

In this section, we verify Conjecture \ref{conj DT/PT} in some compact examples.

${}$ \\
\textbf{Low degree curve classes}.
\begin{lem}
Let $X$ be a smooth projective variety and $\beta\in H_{2}(X)$. If 
any element $C$ of the Chow variety $\mathrm{Chow}_{\beta}(X)$ satisfies~$\chi(\oO_C)=1$, then 
$$I_{0}(X,\beta)=P_{0}(X,\beta)=\varnothing, $$ 
$$I_{1}(X,\beta)\cong P_{1}(X,\beta). $$
Moreover, when $X$ is a Calabi-Yau 4-fold and $\beta$ is as above, Conjecture \ref{conj DT/PT} (1) and (2) for $n=1$ hold.
\end{lem}
\begin{proof}
Let $n=0$ or $1$. Given $[I_Z] \in I_{n}(X,\beta)$ with torsion filtration $T_0\subseteq \oO_Z$ such that 
$\oO_C:=\oO_Z/T_0$ with $C$ Cohen-Macaulay, we have 
\begin{align*}n=\chi(\oO_Z)=\chi(\oO_C)+\chi(T_0)\geqslant \chi(\oO_C)=1. \end{align*}
This is only possible when $n=1$ and $T_0=0$, so $I_Z$ is a stable pair. 

Similarly for $[I^\mdot=\{\oO_X\to F \}]\in P_n(X,\beta)$, we have an exact sequence 
\begin{align*}0\to \oO_C \to F \to Q \to 0, \end{align*}
where $C=\mathrm{supp}(F)$ and $Q$ is 0-dimensional. This is only possible when $n=1$ and $Q=0$, so $I^\mdot\cong I_C$. 
Working in families, we obtain an isomorphism 
\begin{align*}I_1(X,\beta)\cong P_1(X,\beta). \end{align*}
When $X$ is a Calabi-Yau 4-fold, the virtual classes are identified under this isomorphism. Capping with insertions gives the result.\end{proof}
The above lemma enables us to verify our conjecture in the following examples:
\begin{prop}
Conjecture \ref{conj DT/PT} (1) and (2) for $n=1$ hold in the following cases:  
\begin{enumerate}
\item[(A)] $\beta\in H_{2}(X)$ is irreducible and $X$ is one of the following: 
\begin{enumerate}
\item $X$ is one of the quintic fibrations in \cite{KP}, 
\item $X$ is a complete intersection in a product of projective spaces, 
\item $X=Y\times E$ and $\beta\in H_2(Y)\subseteq H_2(X)$, where $Y$ is a smooth complete intersection Calabi-Yau 3-fold in a product of projective spaces and $E$ an elliptic curve,
\end{enumerate}
\item[(B)] $\beta=2[l]$ and $X\subseteq \mathbb{P}^5$ is a generic sextic 4-fold, where $[l] \in H_2(X)$ is the class of a line. 
\end{enumerate}
\end{prop}
\begin{proof}
In case (B), any $C\in \Chow_{\beta}(X)$ is a smooth conic or union of two distinct lines \cite[Prop.~1.4]{Cao2} hence $\chi(\oO_C) = 1$.
The other cases are obvious (see also \cite[Proposition 2.8]{CMT2}).
\end{proof}

${}$ \\
\textbf{Elliptic fibrations}.
Next, we discuss non-primitive examples on elliptic fibrations. For $Y=\mathbb{P}^3$, we take general elements
\begin{align*}
u \in H^0(Y, \oO_Y(-4K_Y)), \
v \in H^0(Y, \oO_Y(-6K_Y)).
\end{align*}
Let $X$
be a smooth projective Calabi-Yau 4-fold with an elliptic fibration
\begin{align}\label{elliptic fib}
\pi \colon X \to Y
\end{align}
given by the Weierstrass equation
\begin{align*}
zy^2=x^3 +uxz^2+vz^3
\end{align*}
in the $\mathbb{P}^2$-bundle
\begin{align*}
\mathbb{P}(\oO_Y(-2K_Y) \oplus \oO_Y(-3K_Y) \oplus \oO_Y) \to Y,
\end{align*}
where $[x:y:z]$ are homogeneous coordinates in the fibres of the above
projective bundle. A general fiber of
$\pi$ is a smooth elliptic curve, and any singular
fiber is either a nodal or cuspidal plane curve.
Moreover, $\pi$ admits a section $\iota : Y \hookrightarrow X$ whose image
corresonds to the fibre point $[0: 1: 0]$. We denote the class of the fibre by $[f] \in H_2(X)$.
\begin{lem}\label{lem elliptic fib}
For any $r \geqslant1$, there exists an isomorphism
\begin{equation}I_{0}(X,r[f])\cong \Hilb^{r}(\mathbb{P}^{3}),  \nonumber \end{equation}
under which the virtual class is given by
\begin{equation}[I_{0}(X,r[f])]^{\rm{vir}}=(-1)^r\cdot[\Hilb^{r}(\mathbb{P}^{3})]^{\rm{vir}},  \nonumber \end{equation}
for certain choice of orientation on the LHS. 
\end{lem}
\begin{proof}
The proof is similar to the stable pair case \cite[Lemma 2.5]{CMT2}. 
We show that the natural morphism
\begin{align*}
\pi^{\ast} \colon \Hilb^r(\mathbb{P}^3) \to I_0(X, r[f])
\end{align*}
is an isomorphism.
Let $[I_Z] \in I_0(X, r[f])$ be an ideal sheaf and $T_0\subseteq \oO_Z$ the torsion subsheaf. 
Denote $\oO_C:=\oO_Z/T_0$, then
\begin{align} \label{ineqs}
0=\chi(\oO_Z)=\chi(\oO_C)+\chi(T_0)\geqslant \chi(\oO_C). 
\end{align}
Fix a polarization on $X$. By the Harder-Narasimhan and Jordan-H\"older filtrations, we have 
\begin{equation}0=F_{0}\subseteq F_{1}\subseteq F_{2}\subseteq\cdot\cdot\cdot\subseteq F_{n}=\oO_C,  \nonumber \end{equation}
where the quotients $E_{i}=F_{i}/F_{i-1}$ are 1-dimensional and stable with decreasing reduced Hilbert polynomials
\begin{equation}p(E_{1})\geqslant p(E_{2})\geqslant \cdot\cdot\cdot\geqslant p(E_{n}).  \nonumber \end{equation}
Note $\ch(E_{i})=(0,0,0,r_{i}[f],\chi(E_{i}))$ for some $r_{i}\geqslant1$, 
so the above inequalities are equivalent to
\begin{equation}\frac{\chi(E_{1})}{r_{1}}\geqslant \frac{\chi(E_{2})}{r_{2}}\geqslant \cdots\geqslant \frac{\chi(E_{n})}{r_{n}}.  \nonumber \end{equation}
By \cite[Lem.~2.2]{CMT1}, stability of $E_{i}$ implies that it is scheme theoretically supported on some fibre $X_{p_{i}}=\pi^{-1}(p_{i})$ of $\pi$, i.e.~$E_{i}=(\iota_{p_{i}})_{*} (E_{i}')$ for some $\iota_{p_{i}}:X_{p_{i}}\hookrightarrow X$ and stable sheaf $E_{i}'\in \Coh(X_{p_{i}})$.

Since $s:\mathcal{O}_{X}\rightarrow \oO_C$ is surjective, so is the composition $\mathcal{O}_{X}\rightarrow \oO_C\twoheadrightarrow E_{n}$.
By adjunction, we obtain an isomorphism
\begin{equation}\Hom_{X}(\mathcal{O}_{X},E_{n})\cong \Hom_{X_{p_{n}}}(\mathcal{O}_{X_{p_{n}}},E_{n}')\neq0, \nonumber \end{equation}
which implies that $p(E_{n}')\geqslant 0$, hence $\chi(E_{n})\geqslant 0$. Therefore
\begin{equation}0\geqslant\chi(\oO_C)=\sum_{i=1}^{n}\chi(E_{i})\geqslant 0,    \nonumber \end{equation}
which implies $T_0=0$ and $I_Z=I_C$ is a stable pair by \eqref{ineqs}. The remaining argument is the same as the stable pair case (ref.~\cite[Lemma 2.5]{CMT2}). 
\end{proof}
In particular, we conclude:
\begin{prop}
Conjecture \ref{conj DT/PT} (1) holds in the following cases: 
\begin{enumerate}
\item $X$ the Weierstrass elliptic fibration (\ref{elliptic fib}) and $\beta=r[f]$ for any $r\geqslant1$.  
\item  $X=Y\times E$ is the product of a smooth projective Calabi-Yau 3-fold $Y$ and an elliptic curve $E$, and $\beta=r[E]$, for any $r\geqslant1$, where $[E]$ denotes the class of $\{\mathrm{pt}\} \times E$.
\end{enumerate}
\end{prop}
\begin{proof}
(1) By Lemma \ref{lem elliptic fib} and \cite[Lem. 2.5]{CMT2}, $P_{0,\beta}(X)=I_{0,\beta}(X)$. 
(2) A similar result as Lemma \ref{lem elliptic fib} holds in this case. By comparing with 
\cite[Lem. 2.14]{CMT2}, we are done. 
\end{proof}

\section{$\mathrm{DT/PT}$ on toric Calabi-Yau 4-folds}

In this section, we study equivariant analogues of the DT/PT correspondence for toric Calabi-Yau 4-folds.
After studying fixed loci, we define equivariant curve counting invariants and stable pair invariants, which are rational functions 
of the equivariant parameters. We formulate a vertex formalism for both invariants and conjecture a relation between the (fully equivariant) DT/PT vertex, which we check in several cases. 
This implies a DT/PT correspondence for toric Calabi-Yau 4-folds with primary insertions. For local surfaces, we consider cases where the correspondence reduces to an equality of numbers providing good motivation for Conjecture \ref{conj DT/PT}.

\subsection{Fixed loci of Hilbert schemes} \label{DTfix}

Let $X$ be a smooth quasi-projective toric Calabi-Yau 4-fold. By this we mean a smooth quasi-projective toric 4-fold $X$ satisfying $K_X \cong \oO_X$ and  $H^{>0}(\oO_X) = 0$. We also assume the fan contains cones of dimension 4. Such cones correspond to $(\mathbb{C}^*)^4$-invariant affine open subsets (equivariantly) isomorphic to $\mathbb{C}^4$. Important examples are the following:
\begin{itemize} 
\item $X=\mathrm{Tot}_{\PP^1 \times \PP^1}(\oO(-1,-1) \oplus \oO(-1,-1))$, 
\item $X=\mathrm{Tot}_{\mathbb{P}^2}(\oO(-1) \oplus \oO(-2))$,  
\item $X=\mathrm{Tot}_{\mathbb{P}^1}(\oO(a,b,c))$, with $a+b+c=-2$.
\end{itemize}
Let $\Delta(X)$ be the polytope corresponding to $X$ and denote the collection of its vertices and edges by $V(X)$, $E(X)$ respectively. The elements $\alpha \in V(X)$ correspond to the $(\CC^*)^4$-fixed points $p_\alpha \in X$. Each such fixed point lies in a maximal $(\CC^*)^4$-invariant affine open subset $\CC^4 \cong U_\alpha \subseteq X$. The elements $\alpha\beta \in E(X)$ (connecting vertices $\alpha$, $\beta$) correspond to the $(\CC^*)^4$-invariant lines $\PP^1 \cong C_{\alpha\beta} \subseteq X$. Such a line has normal bundle
\begin{align}
\begin{split} \label{normal}
&N_{C_{\alpha\beta}/X} \cong \oO_{\PP^1}(m_{\alpha\beta}) \oplus  \oO_{\PP^1}(m_{\alpha\beta}') \oplus  \oO_{\PP^1}(m_{\alpha\beta}''), \\
&m_{\alpha\beta} + m_{\alpha\beta}' + m_{\alpha\beta}'' = -2, 
\end{split}
\end{align}
where the first isomorphism follows from Grothendieck's splitting theorem and the second equality follows from the fact that $X$ is Calabi-Yau.

The action of the dense open torus $(\CC^*)^4$, and its Calabi-Yau subtorus $T \subseteq (\CC^*)^4$, lift to the Hilbert scheme $I_n(X,\beta)$. Although $I_n(X,\beta)$ is in general not proper, its $(\CC^*)^4$-fixed and $T$-fixed loci are proper. In fact, these loci coincide and are 0-dimensional and reduced.
\begin{lem} \label{Tfixlocus}
At the level of closed points, we have
\begin{equation}I_n(X,\beta)^T=I_n(X,\beta)^{(\mathbb{C}^*)^4}, \nonumber \end{equation}
which consists of finitely many points.
\end{lem}
\begin{proof}
Similarly as in \cite[Lemma 3.1]{CK} ($\beta=0$ case), we cover $X$ by maximal $(\mathbb{C}^*)^4$-invariant open affine subsets $\{U_\alpha \cong \CC^4\}_{\alpha \in V(X)}$ with centres at $(\mathbb{C}^*)^4$-fixed points $p_\alpha$. There exist coordinates $x_1,x_2,x_3,x_4$ on $U_\alpha\cong\mathbb{C}^4$, such that the action of $t\in(\mathbb{C}^*)^4$ on $U_\alpha$ is given by
\begin{equation}
t \cdot x_i = t_i x_i, \quad \textrm{for all } i=1,2,3,4 \, \textrm{ and } \, t=(t_1,t_2,t_3,t_4) \in (\CC^*)^4. \nonumber 
\end{equation}
Then the Calabi-Yau torus is given by
\begin{equation}T=\{t\in(\mathbb{C}^*)^4\textrm{ }|\textrm{ }t_1t_2t_3t_4=1\}  \nonumber \end{equation}
and we see that $U_\alpha$ is also $T$-invariant. Therefore it suffices to prove the lemma for $X = U_\alpha = \mathbb{C}^4$ with the standard torus action. 

The $(\mathbb{C}^*)^4$-invariant ideals in $\mathbb{C}[x_1,x_2,x_3,x_4]$ are precisely the monomial ideals. Clearly 
\begin{equation}I_n(X,\beta)^T \supseteq I_n(X,\beta)^{(\mathbb{C}^*)^4}. \nonumber \end{equation}
By considering the weight of $x^{n_1}_1 x^{n_2}_2 x^{n_3}_3 x^{n_4}_4$ under the action of $t\in(\mathbb{C}^*)^4$, it is easy to see that
any $T$-invariant ideal $I\subseteq \mathbb{C}[x_1,x_2,x_3,x_4]$ is of form
\begin{equation}I=\langle x^{n_{11}}_1 x^{n_{12}}_2 x^{n_{13}}_3 x^{n_{14}}_4 f_1(x_1x_2x_3x_4),\ldots, x^{n_{l1}}_1 x^{n_{l2}}_2 x^{n_{l3}}_3 x^{n_{l4}}_4 f_l(x_1x_2x_3x_4) \rangle, \nonumber \end{equation}
where $\{f_i(y)\}$ are polynomials of one variable with constant coefficient 1 and $n_{ij} \in \mathbb{Z}_{\geqslant 0}$. 

Suppose $I$ is $T$-invariant and corresponds to a one-dimensional subscheme $Z$. 
If $x_1x_2x_3x_4\neq 0$ on $Z$, then $f_i(x_1x_2x_3x_4)=0$ on $Z$ for all $i=1, \ldots l$. 
However such a system of equations cannot cut out a one-dimensional subscheme.
So $x_1x_2x_3x_4=0$ on $Z$ and $f_i(x_1x_2x_3x_4)\neq 0$ on $Z$ for all $i=1, \ldots, l$. Hence the open subset
$$
U =  \{f_1(x_1x_2x_3x_4) \neq 0\} \cap \ldots \cap \{f_l(x_1x_2x_3x_4) \neq 0\}
$$
contains $Z$. The polynomials $f_i(x_1x_2x_3x_4)$ become invertible elements after restriction to $U$, hence
$$
I|_U = \left\langle x^{n_{11}}_1 x^{n_{12}}_2 x^{n_{13}}_3 x^{n_{14}}_4,\ldots, x^{n_{l1}}_1 x^{n_{l2}}_2 x^{n_{l3}}_3 x^{n_{l4}}_4 \right\rangle.
$$
We conclude that
$$
I = \left\langle x^{n_{11}}_1 x^{n_{12}}_2 x^{n_{13}}_3 x^{n_{14}}_4,\ldots, x^{n_{l1}}_1 x^{n_{l2}}_2 x^{n_{l3}}_3 x^{n_{l4}}_4 \right\rangle,
$$
which shows $I_n(X,\beta)^T \subseteq I_n(X,\beta)^{(\mathbb{C}^*)^4}$.
\end{proof}

Similarly to \cite[I, Lem.~6]{MNOP} and \cite[Lem.~3.4, 3.6]{CK}, we have the following.
\begin{lem}\label{Tfixlocus scheme}
For any $Z\in I_n(X,\beta)^T$, we have an isomorphism of $T$-representations
\begin{equation}\Ext^0(I_Z,\oO_Z) \cong \Ext^{1}(I_Z,I_Z).  \nonumber \end{equation}
Moreover, $\Ext^0(I_Z,\oO_Z)^T=0$. In particular, the scheme $I_n(X,\beta)^T=I_n(X,\beta)^{(\mathbb{C}^*)^4}$ consists of finitely many \emph{reduced} points.
\end{lem}

We characterize the elements $I_n(X,\beta)^{T}$ by collections of so-called \emph{solid partitions}.
\begin{defi}
A solid partition $\pi$ is a sequence $\pi = \{\pi_{ijk} \in \ZZ_{ \geqslant 0} \cup \{\infty\} \}_{i,j,k \geqslant 1}$ satisfying:
$$
\pi_{ijk} \geqslant \pi_{i+1,j,k}, \qquad \pi_{ijk} \geqslant \pi_{i,j+1,k}, \qquad \pi_{ijk} \geqslant \pi_{i,j,k+1} \qquad \forall \,\, i,j,k \geqslant 1.
$$
This extends the notions of a \emph{plane partition} $\lambda = \{ \lambda_{ij}\}_{i,j \geqslant 1}$ (which we visualize as a pile of boxes in $\RR^3$ where $\lambda_{ij}$ is the height along the $x_3$-axis) and (ordinary) \emph{partitions} $\lambda = \{ \lambda_{i}\}_{i \geqslant 1}$ (which we visualize as a pile of squares in $\RR^2$ where $\lambda_i$ is the height along the $x_2$-axis). Given a solid partition $\pi = \{\pi_{ijk}\}_{i,j,k \geqslant 1}$, there exist unique plane partitions $\lambda, \mu, \nu, \rho$ such that
\begin{align*}
&\pi_{ijk} = \lambda_{jk}, \qquad \forall \,\, i \gg 0,\,\, j,k \geqslant 1 \\
&\pi_{ijk} = \mu_{ik}, \qquad \forall \,\, j \gg 0, \,\, i,k \geqslant 1 \\
&\pi_{ijk} = \nu_{ij}, \,\qquad \forall \,\, k \gg 0,\,\,  i,j \geqslant 1 \\
&\pi_{ijk} = \infty \Leftrightarrow \,\, k = \rho_{ij}, \quad \forall \,\, i,j,k \geqslant 1.
\end{align*}
We refer to $\lambda, \mu, \nu, \rho$ as the \emph{asymptotic plane partitions} associated to $\pi$ in directions $1,2,3,4$ respectively. We call $\pi$ \emph{point-like}, when $\lambda = \mu = \nu = \rho =  \varnothing$. Then the \emph{size} of $\pi$ is defined by
$$
|\pi| := \sum_{i,j,k \geqslant 1} \pi_{ijk}.
$$
We call $\pi$ \emph{curve-like} when $\lambda, \mu, \nu, \rho$ have finite size $|\lambda|, |\mu|, |\nu|, |\rho|$ (not all zero). When $\pi$ is curve-like, we define its \emph{renormalized volume} (similar to \cite{MNOP}) as follows
$$
|\pi| := \sum_{(i,j,k,l) \in \mathbb{Z}_{\geqslant 1}^4 \atop l \leqslant \pi_{ijk}} \big( 1 - \#\{\textrm{legs containing } (i,j,k,l) \} \big).
$$
\end{defi}

Let $[Z] \in I_n(X,\beta)^{T}$. Suppose $\CC^4 \cong U \subseteq X$ is a maximal $(\CC^*)^4$-invariant affine open subset. Then there are coordinates $(x_1,x_2,x_3,x_4)$ such that
\begin{equation}
t \cdot x_i = t_i x_i, \quad \textrm{for all } i=1,2,3,4 \, \textrm{ and } \, t=(t_1,t_2,t_3,t_4) \in (\CC^*)^4. \nonumber 
\end{equation}
The restriction $Z|_U$ is cut out by a $(\CC^*)^4$-invariant ideal $I_Z|_U \subseteq \CC[x_1,x_2,x_3,x_4]$. Solid partitions $\pi$ (point- or curve-like) are in bijective correspondence to $(\CC^*)^4$-invariant ideal $I_{Z_\pi} \subseteq \CC[x_1,x_2,x_3,x_4]$ cutting out subschemes $Z_\pi \subseteq \CC^4$ of dimension $\leqslant 1$ via the following formula
\begin{equation} \label{partitionideal}
I_{Z_\pi} = \left\langle x_1^{i-1} x_2^{j-1} x_3^{k-1} x_4^{\pi_{ijk}} \, : \, i,j,k \geqslant 1  \textrm{ such that } \pi_{ijk} < \infty \right\rangle.
\end{equation}
Therefore, $[Z] \in I_n(X,\beta)^{T}$ determines a collection of solid partitions $\{\pi^{(\alpha)}\}_{\alpha = 1}^{e(X)}$, where $e(X)$ is the topological Euler characteristic of $X$, which equals the number of $(\CC^*)^4$-fixed points of $X$. Let $\alpha\beta \in E(X)$ and consider the corresponding $(\CC^*)^4$-invariant line $C_{\alpha\beta} \cong \PP^1$. Suppose this line given by $\{x_2=x_3=x_4=0\}$ in both charts $U_\alpha$, $U_\beta$. Let $\lambda^{(\alpha)}$, $\lambda^{(\beta)}$ be the asymptotic plane partitions of $\pi^{(\alpha)}$, $\pi^{(\beta)}$ along the $x_1$-axes in both charts. Then
\begin{equation} \label{glue}
\lambda^{(\alpha)}_{ij} = \lambda^{(\beta)}_{ij} =: \lambda_{ij}^{(\alpha\beta)} \qquad \forall \,\, i,j \geqslant 1. 
\end{equation}
When a collection of (point- or curve-like) solid partitions $\{\pi^{(\alpha)}\}_{\alpha = 1}^{e(X)}$ satisfies \eqref{glue}, for all $\alpha, \beta = 1, \ldots, e(X)$, we say that it satisfies the \emph{gluing condition}. Therefore, we obtain a bijective correspondence  
\begin{align*}
&\left\{ {\boldsymbol{\pi}} = \{\pi^{(\alpha)}\}_{\alpha = 1}^{e(X)} \, : \,  \pi^{(\alpha)} \textrm{ \, point- or curve-like and satisfying \eqref{glue}} \right\} \\
&\qquad \qquad \qquad \stackrel{1-1}{\longleftrightarrow}  \\  
&Z_{{\boldsymbol{\pi}}} \in \bigcup_{\beta \in H_2(X), \, n \in \ZZ} I_n(X,\beta)^{T}.
\end{align*}

Suppose $\beta \neq 0$ is effective. Then for any $Z \in I_n(X,\beta)^{T}$, there exists a maximal Cohen-Macaulay subscheme $C \subseteq Z$ such that the cokernel in the short exact sequence
$$
0 \rightarrow I_Z \rightarrow I_C \rightarrow I_C / I_Z \rightarrow 0
$$
is 0-dimensional. The restriction $C|_{U_\alpha}$ is empty or corresponds to a curve-like solid partition $\pi$ with asymptotics $\lambda, \mu, \nu, \rho$. Since $C|_{U_\alpha}$ has no embedded points, the solid partition $\pi$ is entirely determined by the asymptotics $\lambda, \mu ,\nu, \rho$ as follows
\begin{equation} \label{CMsolid}
\pi_{ijk} = \left\{\begin{array}{cc} \infty & \textrm{if \, } 1 \leqslant k \leqslant \rho_{ij} \\ \max\{ \lambda_{jk}, \mu_{ik}, \nu_{ij} \}  & \textrm{otherwise.}  \end{array} \right.
\end{equation}

For any plane partition of finite size, and $m,m',m'' \in \ZZ$, we define
$$
f_{m,m',m''}(\lambda) := \sum_{i,j \geqslant 1} \sum_{k=1}^{\lambda_{ij}} (1-m(i-1) - m'(j-1) - m''(k-1)).
$$
For any $\alpha\beta \in E(X)$ and finite plane partition $\lambda$, we set
$$
f(\alpha,\beta) := f_{m_{\alpha\beta},m_{\alpha\beta}',m_{\alpha\beta}''}(\lambda),
$$
where $m_{\alpha\beta}, m_{\alpha\beta}', m_{\alpha\beta}''$ were defined in \eqref{normal}.
\begin{lem} \label{chi}
Let $X$ be a smooth toric Calabi-Yau 4-fold and $Z \subseteq X$ a $(\CC^*)^4$-invariant closed subscheme of dimension $\leqslant 1$. Then
$$
\chi(\oO_Z)  = \sum_{\alpha \in V(X)} |\pi^{(\alpha)}| + \sum_{\alpha\beta \in E(X)} f(\alpha,\beta).
$$
\end{lem}
\begin{proof}
Denote by $\{\pi^{(\alpha)}\}_{\alpha \in V(X)}$ the collection of solid partitions corresponding to $Z$ and denote the corresponding asymptotic plane partitions by $\{\lambda^{(\alpha\beta)}\}_{\alpha\beta \in E(X)}$. For each $\alpha \in V(X)$, the renormalized volume of $\pi^{(\alpha)}$ is given by
$$
|\pi^{(\alpha)}| = \sum_{(i,j,k,l) \in \pi} \left( 1 - \# \{ \textrm{legs containing } (i,j,k,l) \} \right),
$$
where $(i,j,k,l) \in \pi^{(\alpha)}$ means $i,j,k,l  \in \ZZ_{\geqslant 1}$ and $l \leqslant \pi^{(\alpha)}_{ijk}$.

Suppose $C \subseteq Z$ is the maximal Cohen-Macaulay subcurve of $Z$. Each finite plane partition $\lambda^{(\alpha\beta)}$ determines an associated Cohen-Macaulay curve on $X$, which is the thickening of $C_{\alpha\beta} \cong \PP^1$ according to the plane partitions $\lambda^{(\alpha\beta)}$. We denote this Cohen-Macaulay curve by $\lambda^{(\alpha\beta)} C_{\alpha\beta} \subseteq X$. Its underlying reduced curve is $C_{\alpha\beta} \cong \PP^1$. Then 
$$
C \subseteq \bigcup_{\alpha\beta \in E(X)} \lambda^{(\alpha\beta)} C_{\alpha\beta}
$$ 
and by inclusion-exclusion, we have
$$
\chi(\oO_C) = \sum_{\alpha\beta \in E(X)} \chi(\oO_{\lambda^{(\alpha\beta)} C_{\alpha\beta}}) + \sum_{\alpha \in V(X)} \sum_{(i,j,k,l) \in \pi^{(\alpha)} \atop (i,j,k,l) \textrm{ in a leg of } \pi^{(\alpha)}} (1- \# \{ \textrm{legs containing} \, (i,j,k,l) \}).
$$
Furthermore, we have
\begin{align*}
\chi(\oO_{\lambda^{(\alpha\beta)} C_{\alpha\beta}}) &= \sum_{(i,j,k) \in \lambda^{(\alpha\beta)}} \chi(\oO_{\PP^1} (-(i-1) m_{\alpha\beta}) \otimes \oO_{\PP^1}(-(j-1) m'_{\alpha\beta}) \otimes \oO_{\PP^1}( -(k-1) m''_{\alpha\beta})) \\
&= \sum_{(i,j,k) \in \lambda^{(\alpha\beta)}} (1 - (i-1) m_{\alpha\beta} - (j-1) m'_{\alpha\beta} - (k-1) m''_{\alpha\beta}).
\end{align*}
The lemma follows from the fact that the kernel of $\oO_Z \twoheadrightarrow \oO_C$ equals $I_C / I_Z$ and
\begin{equation*}
\chi(I_C / I_Z) = \sum_{\alpha \in V(X)} \chi(I_{C|_{U_\alpha}} / I_{Z|_{U_\alpha}} ) = \sum_{\alpha \in V(X)} \sum_{(i,j,k,l) \in \pi^{(\alpha)} \atop  (i,j,k,l) \textrm{ not in a leg of } \pi^{(\alpha)}} 1. \hfill \qedhere
\end{equation*}
\end{proof}

\subsection{Fixed loci of moduli spaces of stable pairs} \label{PTfix}

Let $X$ be a toric Calabi-Yau 4-fold. Then the action of $(\CC^*)^4$ lifts to the stable pairs space $P_n(X,\beta)$. We will describe the fixed loci $P_n(X,\beta)^{(\CC^*)^4}$ and $P_n(X,\beta)^{T}$ closely following Pandharipande-Thomas \cite{PT2}. 

Given a stable pair $(F,s)$ on $X$, the scheme-theoretic support $C_F := \mathrm{supp}(F)$ of $F$ is a Cohen-Macaulay curve \cite[Lem.~1.6]{PT1}. Stable pairs with Cohen-Macaulay support curve $C$ can be described as follows \cite[Prop.~1.8]{PT1}: \\

\noindent Let $\mathfrak{m} \subseteq \oO_C$ be the ideal of a finite union of closed points on $C$. A stable pair $(F,s)$ such that $C_F = C$ and $\mathrm{supp}(Q)_{\mathrm{red}} \subseteq \mathrm{supp}(\oO_C / \mathfrak{m})$ is equivalent to a subsheaf of $\varinjlim \HOM(\mathfrak{m}^r, \oO_C) / \oO_C$. \\

This uses the natural inclusions
\begin{align*}
\HOM(\mathfrak{m}^r, \oO_C) \hookrightarrow \HOM(\mathfrak{m}^{r+1}, \oO_C) \\
\oO_C \hookrightarrow \HOM(\mathfrak{m}^r, \oO_C)
\end{align*}
induced by $\mathfrak{m}^{r+1} \subseteq \mathfrak{m}^{r} \subseteq \oO_C$. 

Suppose $[(F,s)] \in P_n(X,\beta)^{(\CC^*)^4}$, then $C_F$ is $(\CC^*)^4$-fixed and therefore determines solid partitions $\{\pi^{(\alpha)}\}_{\alpha \in V(X)}$ as in the previous section. Consider a maximal $(\CC^*)^4$-invariant affine open subset $\CC^4 \cong U_\alpha \subseteq X$. Denote the asymptotic plane partitions of $\pi := \pi^{(\alpha)}$ in directions $1,2,3,4$ by $\lambda, \mu, \nu, \rho$. Just like in \eqref{partitionideal}, these correspond to $(\CC^*)^4$-invariant ideals
\begin{align*}
I_{Z_\lambda} &\subseteq \CC[x_2,x_3,x_4], \\
I_{Z_\mu} &\subseteq \CC[x_1,x_3,x_4], \\
I_{Z_\nu} &\subseteq \CC[x_1,x_2,x_4], \\
I_{Z_\rho} &\subseteq \CC[x_1,x_2,x_3].
\end{align*}
Define the following $\CC[x_1,x_2,x_3,x_4]$-modules
\begin{align*}
M_1 &:= \CC[x_1,x_1^{-1}] \otimes_{\CC} \CC[x_2,x_3,x_4] / I_{Z_\lambda}, \\
M_2 &:= \CC[x_2,x_2^{-1}] \otimes_{\CC} \CC[x_1,x_3,x_4] / I_{Z_\mu}, \\
M_3 &:= \CC[x_3,x_3^{-1}] \otimes_{\CC} \CC[x_1,x_2,x_4] / I_{Z_\nu}, \\
M_4 &:= \CC[x_4,x_4^{-1}] \otimes_{\CC} \CC[x_1,x_2,x_3] / I_{Z_\rho}.
\end{align*}
Then \cite[Sect.~2.4]{PT2}
$$
\varinjlim \HOM(\mathfrak{m}^r, \oO_{C|_{U_\alpha}})  \cong \bigoplus_{i=1}^4 M_i =: M,
$$
where $\mathfrak{m} = \langle x_1, x_2, x_3, x_4 \rangle \subseteq \CC[x_1,x_2,x_3,x_4]$. Each module $M_i$ comes from a ring and therefore has a unit 1, which is homogeneous of degree $(0,0,0,0)$ with respect to the character group $X((\CC^*)^4) = \ZZ^4$. We form the quotient
\begin{equation} \label{Mmod}
M / \langle (1,1,1,1) \rangle.
\end{equation}
Then $(\CC^*)^4$-equivariant stable pairs on $U_\alpha \cong \CC^4$ correspond to $(\CC^*)^4$-invariant $\CC[x_1,x_2,x_3,x_4]$-submodules of \eqref{Mmod}. \\

\noindent \textbf{Combinatorial description $M / \langle (1,1,1,1) \rangle$}. Denote the character group of $(\CC^*)^4$ by $X((\CC^*)^4) = \ZZ^4$. For each module $M_i$, the weights $w \in \ZZ^4$ of its non-zero eigenspaces determine an infinite ``leg'' $\mathrm{Leg}_i \subseteq \ZZ^4$ along the $x_i$-axis. For each weight $w \in \ZZ^4$, introduce four independent vectors $\boldsymbol{1}_w$, $\boldsymbol{2}_w$, $\boldsymbol{3}_w$, $\boldsymbol{4}_w$. Then the $\CC[x_1,x_2,x_3,x_4]$-module structure on $M / \langle (1,1,1,1) \rangle$ is determined by
$$
x_j \cdot \boldsymbol{i}_{w} = \boldsymbol{i}_{w+e_j},
$$
where $i,j=1, \ldots, 4$ and $e_1, \ldots, e_4$ are the standard basis vectors of $\ZZ^4$. Similar to the 3-fold case \cite[Sect.~2.5]{PT2}, we define regions
$$
\mathrm{I}^+ \cup \mathrm{II} \cup \mathrm{III} \cup \mathrm{IV} \cup \mathrm{I}^- = \bigcup_{i=1}^4 \mathrm{Leg}_i  \subseteq \ZZ^4,
$$
as follows:
\begin{itemize}
\item $\mathrm{I}^+$ consists of all weights $w \in \ZZ^4$ with all coordinates non-negative \emph{and} which lie in precisely one leg. If $w \in \mathrm{I}^+$, then the corresponding weight space of $M / \langle (1,1,1,1) \rangle$ is 0-dimensional.
\item $\mathrm{I}^-$ consists of all weights $w \in \ZZ^4$ with at least one negative coordinate. If $w \in \mathrm{I}^-$ is supported in $\mathrm{Leg}_i$, then the corresponding weight space of $M / \langle (1,1,1,1) \rangle$ is 1-dimensional
$$
\CC \cong \CC \cdot \boldsymbol{i}_w \subseteq M / \langle (1,1,1,1) \rangle.
$$
\item $\mathrm{II}$ consists of all weights $w \in \ZZ^4$, which lie in precisely two legs. If $w \in \mathrm{II}$ is supported in $\mathrm{Leg}_i$ and $\mathrm{Leg}_j$, then the corresponding weight space of $M / \langle (1,1,1,1) \rangle$ is 1-dimensional
$$
\CC \cong \CC \cdot \boldsymbol{i}_w \oplus \CC \cdot \boldsymbol{j}_w / \CC \cdot (\boldsymbol{i}_w + \boldsymbol{j}_w) \subseteq M / \langle (1,1,1,1) \rangle.
$$
\item $\mathrm{III}$ consists of all weights $w \in \ZZ^4$, which lie in precisely three legs. If $w \in \mathrm{III}$ is supported in $\mathrm{Leg}_i$, $\mathrm{Leg}_j$, and $\mathrm{Leg}_k$, then the corresponding weight space of $M / \langle (1,1,1,1) \rangle$ is 2-dimensional
$$
\CC^2 \cong \CC \cdot \boldsymbol{i}_w \oplus \CC \cdot \boldsymbol{j}_w \oplus \CC \cdot \boldsymbol{k}_w / \CC \cdot (\boldsymbol{i}_w + \boldsymbol{j}_w + \boldsymbol{k}_w) \subseteq M / \langle (1,1,1,1) \rangle.
$$
\item $\mathrm{IV}$ consists of all weights $w \in \ZZ^4$, which lie in all four legs. If $w \in \mathrm{IV}$, then the corresponding weight space of $M / \langle (1,1,1,1) \rangle$ is 3-dimensional
$$
\CC^3 \cong \CC \cdot \boldsymbol{1}_w \oplus \CC \cdot \boldsymbol{2}_w \oplus \CC \cdot \boldsymbol{3}_w \oplus \CC \cdot \boldsymbol{4}_w / \CC \cdot (\boldsymbol{1}_w + \boldsymbol{2}_w + \boldsymbol{3}_w + \boldsymbol{4}_w) \subseteq M / \langle (1,1,1,1) \rangle.
$$
\end{itemize}

\noindent \textbf{Box configurations}.  A box configuration is a finite collection of weights $B \subseteq \mathrm{II} \cup \mathrm{III} \cup \mathrm{IV} \cup \mathrm{I}^-$ satisfying the following property: \\

\noindent if $w = (w_1,w_2,w_3,w_4) \in \mathrm{II} \cup \mathrm{III} \cup \mathrm{IV} \cup \mathrm{I}^-$ and one of $(w_1-1,w_2,w_3,w_4)$, $(w_1,w_2-1,w_3,w_4)$, $(w_1,w_2,w_3-1,w_4)$, or $(w_1,w_2,w_3,w_4-1)$ lies in $B$ then $w \in B$.\footnote{In J.~Bryan's words: gravity pulls in the $(1,1,1,1)$-direction.} \\

Such a box configuration determines a $(\CC^*)^4$-invariant submodule of $M / \langle (1,1,1,1) \rangle$ and therefore a $(\CC^*)^4$-invariant stable pair on $U_\alpha \cong \CC^4$ with cokernel of length
$$
\#  (B \cap \mathrm{II}) + 2 \cdot \#  (B \cap \mathrm{III}) + 3 \cdot \#  (B \cap \mathrm{IV}) + \#  (B \cap \mathrm{I}^-).
$$

The box configurations defined in this section by no means describe \emph{all} $(\CC^*)^4$-invariant submodules of $M / \langle (1,1,1,1) \rangle$. From now on, we restrict attention to the case where the fixed loci $P_n(X,\beta)^{(\CC^*)^4}$ are at most 0-dimensional for all $n$. Then the restriction of any $[(F,s)] \in P_n(X,\beta)^{(\CC^*)^4}$ to each $U_\alpha$ has a Cohen-Macaulay support curve with at most two asymptotic plane partitions (as we prove below) and is described by a box configuration as above.

\begin{prop} \label{fewlegs}
Suppose $P_n(X,\beta)^{(\CC^*)^4}$ is at most 0-dimensional for all $n \in \ZZ$. Then for any $[(F,s)] \in P_n(X,\beta)^{(\CC^*)^4}$ and any $\alpha \in V(X)$, the Cohen-Macaulay curve $C_F |_{U_\alpha}$ has at most two asymptotic plane partitions.
\end{prop}
\begin{proof}
Suppose there is an $\alpha \in V(X)$, such that the module $M / \langle (1,1,1,1) \rangle$ associated to $C_F |_{U_\alpha}$, for some $\alpha \in V(X)$, has three or four non-empty asymptotic plane partitions. \\

\noindent \textbf{Case 1:} $C_F |_{U_\alpha}$ has four non-empty asymptotic plane partitions. Construct the following $(\CC^*)^4$-invariant submodule of $M / \langle (1,1,1,1) \rangle$. At $w=(0,0,0,0)$, choose any 1-dimensional subspace 
$$
V_0 \subseteq \CC \cdot \boldsymbol{1}_0 \oplus \CC \cdot \boldsymbol{2}_0 \oplus \CC \cdot \boldsymbol{3}_0 \oplus \CC \cdot \boldsymbol{4}_0 / \CC \cdot (\boldsymbol{1}_0 + \boldsymbol{2}_0 + \boldsymbol{3}_0 + \boldsymbol{4}_0)
$$ 
and for any $w \in (\mathrm{IV} \cup \mathrm{III} \cup \mathrm{II} \cup \mathrm{I}^-) \setminus \{0\}$, we define 
$$
V_w := \left\{ \begin{array}{cc} (M / \langle (1,1,1,1) \rangle)_w & \textrm{if } w_1,w_2,w_3,w_4 \geqslant 0 \\ 0 & \textrm{otherwise} \end{array} \right.
$$
This determines a $(\CC^*)^4$-invariant submodule of $M / \langle (1,1,1,1) \rangle$. By varying $V_0$ over all possible 1-dimensional subspaces, we obtain a connected component of $\bigcup_n P_n(X,\beta)^{(\CC^*)^4}$ with $\PP^2$ as its underlying reduced variety. Contradiction. \\

\noindent \textbf{Case 2:} $C_F |_{U_\alpha}$ has three non-empty asymptotic plane partitions, say in directions 1,2,3. The same strategy as above works, this time choosing a 1-dimensional subspace 
$$
V_0 \subseteq \CC \cdot \boldsymbol{1}_0 \oplus \CC \cdot \boldsymbol{2}_0 \oplus \CC \cdot \boldsymbol{3}_0 / \CC \cdot (\boldsymbol{1}_0 + \boldsymbol{2}_0 + \boldsymbol{3}_0)
$$ 
and setting $V_w = (M / \langle (1,1,1,1) \rangle)_w$ for all $w  \in (\mathrm{III} \cup \mathrm{II} \cup  \mathrm{I}^-) \setminus \{0\}$ with $w_1,w_2,w_3,w_4 \geqslant 0$, and $V_w=0$ otherwise. We obtain a connected component of $\bigcup_n P_n(X,\beta)^{(\CC^*)^4}$ with underlying reduced variety $\PP^1$. Contradiction.
\end{proof}

\begin{prop}  \label{PTfixedlocus}
Suppose $P_n(X,\beta)^{(\CC^*)^4}$ is at most 0-dimensional for all $n \in \ZZ$. Then, as schemes, $P_n(X,\beta)^{T} = P_n(X,\beta)^{(\CC^*)^4}$ consists of finitely many reduced points for all $n \in \ZZ$.
\end{prop}
\begin{proof}
Suppose $[I^\mdot = \{ \oO_X \rightarrow F\}] \in P_n(X,\beta)^{(\CC^*)^4}$. As before, we denote the open cover by maximal $(\CC^*)^4$-invariant affine open subsets by $\{U_\alpha\}_{\alpha \in V(X)}$. Let $I_\alpha^\mdot := I^\mdot|_{U_\alpha}$, $F_\alpha = F|_{U_\alpha}$, and $Q_\alpha = Q|_{U_\alpha}$ (where $Q$ is the cokernel of the stable pair). The proof consists of three parts, which closely follows \cite[Sect.~3.1, 3.2]{PT2}, but now one dimension higher. \\

\noindent \textbf{Step 1:}  The restriction map
\begin{equation} \label{restr}
\Hom_X(I^\mdot,F) \rightarrow \bigoplus_{\alpha \in V(X)} \Hom_{U_\alpha}(I_\alpha^\mdot,F_\alpha)
\end{equation}
is an isomorphism after taking $(\CC^*)^4$-fixed parts. This follows from sequence \cite[(3.2)]{PT2} and the fact that $\Hom_{U_\alpha}(I_{C_\alpha}, F_\alpha)^{(\CC^*)^4} = 0$. The latter vanishing is obvious because there are no $(\CC^*)^4$-weights $w$ such that the corresponding weight spaces $(I_{C_\alpha})_w$ and $(F_\alpha)_w$ are \emph{both} non-zero.

The restriction map \eqref{restr} is \emph{also} an isomorphism after taking $T$-fixed parts. This requires the less obvious fact that $\Hom_{U_\alpha}(I_{C_\alpha}, F_\alpha)^{T} = 0$. The latter vanishing follows from the analog of \cite[Lem.~2]{PT2}, the proof of which immediately adapts to our setting (and is very similar to the manipulation with Haiman arrows detailed in \cite[Lem.~3.6]{CK}). \\

\noindent \textbf{Step 2:} There are isomorphisms $$\Hom_{U_\alpha}(I_\alpha^\mdot,F_\alpha)^{(\CC^*)^4} \cong \Ext_{U_\alpha}^1(Q_\alpha, F_\alpha)^{(\CC^*)^4} \cong \Hom_{U_\alpha}(Q_\alpha, M / F_\alpha)^{(\CC^*)^4}$$ by \cite[(3.2) \& Lem.~1]{PT2}. Since we have at most two legs coming together (Proposition \ref{fewlegs}), there are no $(\CC^*)^4$-weights $w$ such that the weight spaces $(Q_\alpha)_w$ and $(M / F_\alpha)_w$ are both non-zero. Therefore $$\Hom_{U_\alpha}(Q_\alpha, M / F_\alpha)^{(\CC^*)^4} = 0.$$ Note that we do not require the analogue of \cite[Prop.~4]{PT2}, which deals with the more complicated case of labelled boxes (boxes with moduli). Hence $\Hom_X(I^\mdot,F)^{(\CC^*)^4} = 0$ and $P_n(X,\beta)^{(\CC^*)^4}$ consists of finitely many reduced points. \\

\noindent \textbf{Step 3:} For the Calabi-Yau torus $T \subseteq (\CC^*)^4$, we also have $$\Hom_{U_\alpha}(I_\alpha^\mdot,F_\alpha)^{T} \cong \Ext_{U_\alpha}^1(Q_\alpha, F_\alpha)^{T} \cong \Hom_{U_\alpha}(Q_\alpha, M / F_\alpha)^{T}$$ by \cite[Lem.~3]{PT2}. Decompose $\Hom_{U_\alpha}(Q_\alpha, M / F_\alpha)^{T}$ with respect to the 1-dimensional torus $(\CC^*)^4 / T \cong \CC^*$. Suppose there exists a non-zero homogeneous morphism $\phi \in \Hom_{U_\alpha}(Q_\alpha, M / F_\alpha)^{T}$ of weight $n \in \ZZ$. Since we have at most two legs coming together (Corollary \ref{fewlegs}), we must have $n<0$. Since $\phi$ is assumed non-zero, it sends some monomial $x_1^{w_1}x_2^{w_2}x_3^{w_3}x_4^{w_4}$ of a box of $Q_\alpha$ to $\lambda x_1^{w_1+n}x_2^{w_2+n}x_3^{w_3+n}x_4^{w_4+n} \in M / F_\alpha$ for $w$ in some leg, say $\mathrm{Leg}_1$, and some $\lambda \neq 0$. Consider such a $x_1^{w_1}x_2^{w_2}x_3^{w_3}x_4^{w_4}$ with \emph{minimal} $w_1$ and \emph{maximal} $w_2$. Case 1: $Q_\alpha$ supports a box at $(w_1+n,w_2+n+1,w_3+n,w_4+n)$. This is not possible by minimality of $w_1$ (recall $n<0$). Case 2: $Q_\alpha$ does not support a box at $(w_1+n,w_2+n+1,w_3+n,w_4+n)$. Then
\begin{align*}
\lambda x_1^{w_1+n}x_2^{w_2+n+1}x_3^{w_3+n}x_4^{w_4+n} &= x_2 \cdot \phi(x_1^{w_1}x_2^{w_2}x_3^{w_3}x_4^{w_4} ) \\
&= \phi(x_1^{w_1}x_2^{w_2+1}x_3^{w_3}x_4^{w_4} ) \\
&= \phi(0) = 0,
\end{align*}
where $x_1^{w_1}x_2^{w_2+1}x_3^{w_3}x_4^{w_4} = 0 \in Q_\alpha$ by maximality of $w_2$. This is a contradiction because $(w_1+n,w_2+n+1,w_3+n,w_4+n)$ does not support a box of $Q_\alpha$, so $0 \neq x_1^{w_1+n}x_2^{w_2+n+1}x_3^{w_3+n}x_4^{w_4+n} \in M / F_\alpha$.
\end{proof}

\subsection{Equivariant invariants}

Following \cite{CL, CK, CMT2}, we study curve counting and stable pair invariants of toric Calabi-Yau 4-folds $X$. These are defined by $T$-localization, because $X$ is not proper and therefore $I_n(X,\beta)$ is in general not proper either. 

Let $\bullet$ denote $\Spec \mathbb{C}$ with trivial $(\mathbb{C}^*)^4$-action. We denote by $\mathbb{C} \otimes t_i$ the 1-dimensional $(\mathbb{C}^*)^4$-representation with weight $t_i$ and we write $\lambda_i \in H_{(\mathbb{C}^*)^4}^{\ast}(\bullet)$ for its $(\mathbb{C}^*)^4$-equivariant first Chern class. Then
\begin{align*}
H^*_{(\mathbb{C}^*)^4}(\bullet)=\mathbb{Z}[\lambda_1, \lambda_2, \lambda_3,\lambda_4],  \end{align*} 
\begin{align*}
H^*_{T}(\bullet)=\mathbb{Z}[\lambda_1, \lambda_2, \lambda_3,\lambda_4]/\langle \lambda_1+\lambda_2+\lambda_3+\lambda_4 \rangle
\cong \mathbb{Z}[\lambda_1, \lambda_2,\lambda_3]. 
\end{align*} 
Recall that $T \subseteq (\CC^*)^4$ denotes the Calabi-Yau torus, which is defined as the subtorus preserving the Calabi-Yau volume form. Note that the Serre duality pairing $\Ext^i(E,E) \cong \Ext^{4-i}(E,E)^*$ is a $T$-equivariant isomorphism for any $T$-equivariant coherent sheaf $E$ on $X$ and any $i$. 

For each of the finitely many fixed points $Z \in I_n(X,\beta)^T = I_n(X,\beta)^{(\CC^*)^4}$ (Lemma \ref{Tfixlocus scheme}), one can form a complex vector bundle
\begin{equation}
\begin{array}{lll}
      & \quad ET\times_{T}\Ext^{i}(I_Z,I_Z)
      \\  &  \quad \quad\quad \quad \downarrow \\   &   \quad  ET\times_{T}\{I_Z\}=BT
\end{array}\textrm{ }\textrm{for}\textrm{ } i=1,2, \nonumber\end{equation}
whose Euler class is the $T$-equivariant Euler class $e_T\big(\Ext^{i}(I_Z,I_Z)\big)$.

When $i=2$, the Serre duality pairing on $\Ext^{2}(I_Z,I_Z)$ induces a non-degenerate quadratic form $Q$ on $ET\times_{T}\Ext^{2}(I_Z,I_Z)$, because 
$T$ preserves the Calabi-Yau volume form. We define 
\begin{equation}\label{half euler class equivariant}e_T\big(\Ext^{2}(I_Z,I_Z),Q\big)\in\mathbb{Z}[\lambda_1,\lambda_2,\lambda_3] \end{equation}
as the \textit{half Euler class}\,\footnote{I.e.~the Euler class of its positive real form, which is a half rank real subbundle on which $Q$ is real and positive definite.} of $(ET\times_{T}\Ext^{2}(I_Z,I_Z),Q)$, which satisfies 
\begin{align}\label{hal eul cla}(-1)^{\frac{1}{2}\ext^{2}(I_Z,I_Z)}e_T\big(\Ext^{2}(I_Z,I_Z)\big)=\Big(e_T\big(\Ext^{2}(I_Z,I_Z),Q\big)\Big)^2. \end{align}
The half Euler class \eqref{half euler class equivariant} depends on 
the choice of orientation on a positive real form, or ---equivalently--- a choice of square root of (\ref{hal eul cla}), i.e. 
\begin{align*}e_T\big(\Ext^{2}(I_Z,I_Z),Q\big)=\pm\sqrt{(-1)^{\frac{1}{2}\ext^{2}(I_Z,I_Z)}e_T\big(\Ext^{2}(I_Z,I_Z)\big)}\in\mathbb{Z}[\lambda_1,\lambda_2,\lambda_3]. \end{align*}
Similar to \cite[Sect.~3]{CK}, we define equivariant curve counting invariants as follows.
\begin{defi}\label{def of equi curve counting inv} 
Let $X$ be a toric Calabi-Yau 4-fold and $\beta \in H_2(X)$. For $\gamma_1, \ldots, \gamma_m \in H^*_T(X,\QQ)$, we define
\begin{align*}I_{n,\beta}(X;T)(\gamma_1, \ldots, \gamma_m)_{o(\lL)} :=&\sum_{Z\in I_n(X,\beta)^T}
(-1)^{o(\lL)|_Z}\frac{\sqrt{(-1)^{\frac{1}{2}\ext^{2}(I_Z,I_Z)}e_T\big(\Ext^{2}(I_Z,I_Z)\big)}}{e_T\big(\Ext^{1}(I_Z,I_Z)\big)}\cdot
\prod_{i=1}^{m}\tau(\gamma_i)|_{Z} \\
\in&\frac{\QQ(\lambda_1,\lambda_2,\lambda_3,\lambda_4)}{\langle \lambda_1+\lambda_2+\lambda_3+\lambda_4\rangle} \cong \QQ(\lambda_1,\lambda_2,\lambda_3), \end{align*}
where $o(\lL)|_Z$ denotes a choice of sign for each $Z\in I_n(X,\beta)^T$ and $\tau$ is the insertion (\ref{insert}). 
We often drop $o(\lL)$ from the notation. 
When there is no insertion, we write $I_{n,\beta}(X;T)$ for the invariant. We also form the generating series
$$
I_{\beta}(X;T)(\gamma_1, \ldots, \gamma_m):= \sum_{n \in \ZZ} I_{n,\beta}(X;T)(\gamma_1, \ldots, \gamma_m)\, q^n  \in \QQ(\lambda_1,\lambda_2,\lambda_3)(\!(q)\!).
$$
\end{defi}
\begin{rmk}
The number of choices of orientations in the previous definition is $2^{\# I_{n,\beta}(X)^T}$.
\end{rmk}

There is a parallel story for stable pairs. 
For $[I^\mdot = \{\oO_X \rightarrow F\}] \in P_n(X,\beta)^{(\CC^*)^4}$, we can define the equivariant Euler classes $e_T\big(\Ext^{i}(I^\mdot,I^\mdot)\big)$ and a square root
\begin{align*}\sqrt{(-1)^{\frac{1}{2}\ext^{2}(I^\mdot,I^\mdot)}e_T\big(\Ext^{2}(I^\mdot,I^\mdot)\big)}\in\mathbb{Z}[\lambda_1,\lambda_2,\lambda_3]. \end{align*}

We define equivariant stable pair invariants when $P_n(X,\beta)^{(\CC^*)^4}$ is at most 0-dimensional for all $n \in \ZZ$. Then $P_n(X,\beta)^{T} = P_n(X,\beta)^{(\CC^*)^4}$ consists of finitely many reduced points (Proposition \ref{PTfixedlocus}). Recall that this happens in  several interesting examples, e.g.~when $X$ is a local toric curve or surface.
\begin{defi}\label{def of equi pair inv} 
Let $X$ be a toric Calabi-Yau 4-fold and let $\beta \in H_2(X)$. Suppose $P_n(X,\beta)^{(\CC^*)^4}$ is at most 0-dimensional for all $n \in \ZZ$. For $\gamma_1, \ldots, \gamma_m \in H^*_T(X)$, we define
\begin{align*}P_{n,\beta}(X;T)(\gamma_1, \ldots, \gamma_m)_{o(\lL)}:=&\sum_{[I^\mdot]\in P_n(X,\beta)^T}
(-1)^{o(\lL)|_{[I^\mdot]}}\frac{\sqrt{(-1)^{\frac{1}{2}\ext^{2}(I^{\mdot},I^{\mdot})}e_T\big(\Ext^{2}(I^{\mdot},I^{\mdot})\big)}}{e_T\big(\Ext^{1}(I^{\mdot},I^{\mdot})\big)}\cdot
\prod_{i=1}^{m}\tau(\gamma_i)|_{[I^\mdot]} \\
\in&\frac{\QQ(\lambda_1,\lambda_2,\lambda_3,\lambda_4)}{\langle \lambda_1+\lambda_2+\lambda_3+\lambda_4 \rangle} \cong \QQ(\lambda_1, \lambda_2,\lambda_3), \end{align*}
where $o(\lL)|_{[I^\mdot]}$ denotes a choice of sign for each $[I^\mdot]\in P_n(X,\beta)^T$ and $\tau$ is the insertion (\ref{insertion for pt}). 
We often drop $o(\lL)$ from the notation. 
When there is no insertion, we write $P_{n,\beta}(X;T)$ for the invariant. We also form generating series
$$
P_{\beta}(X;T)(\gamma_1, \ldots, \gamma_m):= \sum_{n \in \ZZ} P_{n,\beta}(X;T)(\gamma_1, \ldots, \gamma_m)\, q^n  \in \QQ(\lambda_1,\lambda_2,\lambda_3)(\!(q)\!).
$$
\end{defi}

\begin{rmk}
Unlike the compact case, where $n$ is the virtual dimension, it is \emph{not} true in general that $I_{n,\beta}(X;T)(\gamma_1, \ldots, \gamma_m)$ and $P_{n,\beta}(X;T)(\gamma_1, \ldots, \gamma_m)$ are zero when $n<0$. E.g.~for $X=\mathrm{Tot}_{\PP^1 \times \PP^1}(\oO(-1,-1) \oplus \oO(-1,-1))$, $\beta = (3,2)$, $n=-1$, we have $P_{n,\beta}(X;T) = 0$ only for $2^6$ out of $2^{12}$ possible choices of signs.
\end{rmk}

\subsection{Vertex formalism}\label{vertex section}

We now develop a vertex formalism, which reduces the calculation of invariants of the previous section to $X = \CC^4$. In the case of equivariant curve counting invariants, we closely follow the 3-fold case developed by Maulik-Nekrasov-Okounkov-Pandharipande \cite{MNOP}. For $\beta=0$, this was carried out in \cite{CK, NP}. In the case of equivariant stable pair invariants, we follow the 3-fold case developed by Pandharipande-Thomas \cite{PT2}.

Let $X$ be a toric Calabi-Yau 4-fold and consider the cover $\{U_\alpha\}_{\alpha \in V(X)}$ by maximal $(\CC^*)^4$-invariant affine open subsets. Let $E = I_Z$, with $Z \in I_n(X,\beta)^{(\CC^*)^4}$ or $E = [I^\mdot] \in P_n(X,\beta)^{(\CC^*)^4}$. For each $\alpha \in V(X)$, the scheme $Z|_{U_\alpha}$ corresponds to a solid partition and $I^\mdot |_{U_\alpha} = \{\oO_{U_\alpha} \rightarrow F|_{U_\alpha}\}$ corresponds to a box configuration as described in Sections \ref{DTfix} and \ref{PTfix}. When $E = I^\mdot$, we assume that, for all $\alpha \in V(X)$, the Cohen-Macaulay support curve $C|_{U_\alpha}$ is described by a solid partition with at most two non-empty asymptotic plane partitions. We want to calculate
$$
-\dR\mathrm{Hom}(E,E)_0 \in K_T(\bullet).
$$
In fact, we calculate the class of this complex in $K_{(\CC^*)^4}(\bullet)$. We will use the exact triangle
\begin{equation} \label{ET}
E \rightarrow \oO_X \rightarrow E',
\end{equation}
where $E' = \oO_Z$ when $E = I_Z$, and $E' = F$ when $E = I^\mdot$. In either case, $E'$ is 1-dimensional. Let $\Gamma(-)$ denote the global section functor. Let $U_{\alpha\beta} = U_\alpha \cap U_\beta$, $U_{\alpha\beta\gamma} = U_\alpha \cap U_\beta \cap U_\gamma$, etc., and let $E_\alpha := E|_{U_{\alpha}}$, $E_{\alpha\beta} := E|_{U_{\alpha\beta}}$ etc. 
The local-to-global spectral sequence and calculation of sheaf cohomology with respect to the \v{C}ech cover $\{U_\alpha\}_{\alpha \in V(X)}$ gives
\begin{align*}
-\dR\mathrm{Hom}_X(E,E)_0 &= \sum_{\alpha \in V(X),i} (-1)^i \Big(\Gamma(U_\alpha, \oO_{U_\alpha}) -  \Gamma(U_\alpha, \mathcal{E}{\it{xt}}^i(E_\alpha,E_\alpha)) \Big) \\
&- \sum_{\alpha\beta \in E(X),i} (-1)^i \Big(\Gamma(U_{\alpha\beta}, \oO_{U_{\alpha\beta}}) -  \Gamma(U_{\alpha\beta}, \mathcal{E}{\it{xt}}^i(E_{\alpha\beta},E_{\alpha\beta})) \Big).
\end{align*}
Here we use $H^{>0}(U_\alpha,-) = 0$, because $U_\alpha$ is affine. We also use that intersections $U_{\alpha\beta\gamma \cdots }$, 
with $\alpha, \beta, \gamma, \ldots$ pairwise distinct, do not contribute because
$$
E|_{U_{\alpha\beta\gamma\cdots}} = \oO_{U_{\alpha\beta\gamma\cdots}},
$$
which can be seen from \eqref{ET} and the fact that $E'$ is 1-dimensional. Hence we are reduced to determining
\begin{align*}
-\dR\mathrm{Hom}_{U_\alpha}(E_\alpha,E_\alpha)_0 &= \sum_i (-1)^i \Big( \Gamma(U_\alpha, \oO_{U_\alpha}) -  \Gamma(U_\alpha, \mathcal{E}{\it{xt}}^i(E_\alpha,E_\alpha)) \Big), \\
-\dR\mathrm{Hom}_{U_{\alpha\beta}}(E_{\alpha\beta},E_{\alpha\beta})_0 &= \sum_{i} (-1)^i \Big(\Gamma(U_{\alpha\beta}, \oO_{U_{\alpha\beta}}) -  \Gamma(U_{\alpha\beta}, \mathcal{E}{\it{xt}}^i(E_{\alpha\beta},E_{\alpha\beta})) \Big).
\end{align*}

\noindent \textbf{Vertex term.} On $U_\alpha \cong \mathbb{C}^4$, we use coordinates $x_1,x_2,x_3, x_4$ such that the $(\mathbb{C}^*)^4$-action is 
$$
t \cdot x_i = t_i x_i, \quad \textrm{for all } i=1,2,3,4 \, \textrm{ and } \, t = (t_1,t_2,t_3,t_4) \in (\CC^*)^4. 
$$
Let $R:=\Gamma(\oO_{U_\alpha}) \cong \mathbb{C}[x_1,x_2,x_3,x_4]$. Consider the class $[E_\alpha]$ in the equivariant $K$-group $K_{(\mathbb{C}^*)^4}(U_\alpha)$. There exists a ring isomorphism
$$
K_{(\mathbb{C}^*)^4}(U_\alpha) \cong \mathbb{Z}[t_1^{\pm 1},t_2^{\pm 1},t_3^{\pm},t_4^{\pm 1}],
$$
where $[R]$ corresponds to 1. The Laurent polynomial $\mathsf{P}(E_\alpha)$ corresponding to $[E_\alpha]$ via this isomorphism is called the Poincar\'e polynomial of $E_\alpha$.
For any $w=(w_1,w_2,w_3,w_4) \in X((\CC^*)^4) = \mathbb{Z}^4$, we use multi-index notation
$$
t^w:= t_1^{w_1} t_2^{w_2} t_3^{w_3} t_4^{w_4}.
$$
Then $[R \otimes t^w] \in K_{(\mathbb{C}^*)^4}(U_\alpha)$ corresponds to $t^w \in \mathbb{Z}[t_1^{\pm 1},t_2^{\pm 1},t_3^{\pm 1},t_4^{\pm 1}]$.
 
Define an involution $\overline{(-)}$ on $K_{(\mathbb{C}^*)^4}(U_\alpha)$ by $\mathbb{Z}$-linear extension of
$$
\overline{t^w} := t^{-w}.
$$ 
For any element $M \in K_{(\mathbb{C}^*)^4}(U_\alpha)$, we denote the class of its underlying $(\CC^*)^4$-representation by $\tr_M$.
As in \cite{MNOP}, we take a $(\mathbb{C}^*)^4$-equivariant free resolution
$$
0 \rightarrow F_s \rightarrow \cdots \rightarrow F_0 \rightarrow E_\alpha \rightarrow 0,
$$
where 
$$
F_i = \bigoplus_j R \otimes t^{d_{ij}},
$$
for certain $d_{ij} \in \mathbb{Z}^4$. Then
\begin{equation} \label{Poin}
\mathsf{P}(E_\alpha) = \sum_{i,j} (-1)^i t^{d_{ij}}.
\end{equation}
Denote the $(\mathbb{C}^*)^4$-character of $E'|_{U_\alpha}$ by
$$
Z_\alpha := \tr_{E'|_{U_\alpha}}.
$$
When $E' = \oO_Z$, the scheme $Z|_{U_\alpha}$ corresponds to a solid partition $\pi^{(\alpha)}$ as described in Section \ref{DTfix}. Then
\begin{equation} \label{Zalpha}
Z_\alpha = \sum_{i,j,k\geqslant1} \sum_{l=1}^{\pi_{ijk}^{(\alpha)}} t_1^{i-1} t_2^{j-1} t_3^{k-1} t_4^{l-1}.
\end{equation}
When $E' = F$, we use the short exact sequence
$$
0 \rightarrow \oO_C \rightarrow F \rightarrow Q \rightarrow 0,
$$
where $C$ is the Cohen-Macaulay support curve and $Q$ is the cokernel. Then
\begin{equation} \label{PTZalpha}
Z_\alpha = \tr_{\oO_C|_{U_\alpha}} + \tr_{Q|_{U_\alpha}},
\end{equation}
where $\oO_C|_{U_\alpha}$ is described by a solid partition $\pi^{(\alpha)}$ and $Q|_{U_\alpha}$ is determined by a box configuration $B^{(\alpha)}$ as in Section \ref{PTfix}. Then $\tr_{\oO_C|_{U_\alpha}}$ is given by the RHS of \eqref{Zalpha}. Moreover, $\tr_{Q|_{U_\alpha}}$ is the sum of $t^w$ over all $w \in B^{(\alpha)}$.

In both cases, $Z_\alpha$ can be expressed in terms of the Poincar\'e polynomial of $E_\alpha$ as follows
\begin{equation} \label{ZvP}
Z_\alpha =  \tr_{\oO_{U_\alpha} - E_\alpha} = \frac{1-\mathsf{P}(E_\alpha)}{(1-t_1)(1-t_2)(1-t_3)(1-t_4)}.
\end{equation}
We obtain
\begin{align*}
\dR\mathrm{Hom}(E_\alpha,E_\alpha) &= \sum_{i,j,k,l} (-1)^{i+k} \mathrm{Hom}(R \otimes t^{d_{ij}},R \otimes t^{d_{kl}}) \\
&= \sum_{i,j,k,l} (-1)^{i+k} R \otimes t^{d_{kl} - d_{ij}} \\
&=\mathsf{P}(E_\alpha)\overline{\mathsf{P}(E_\alpha)},  \\
\tr_{\dR\mathrm{Hom}(E_\alpha,E_\alpha)}&= \frac{\mathsf{P}(E_\alpha)\overline{\mathsf{P}(E_\alpha)} }{(1-t_1)(1-t_2)(1-t_3)(1-t_4)},
\end{align*}
where we used \eqref{Poin} for the third equality. Eliminating $\mathsf{P}(E_\alpha)$ by using \eqref{ZvP}, we conclude
\begin{equation} \label{defVprelim}
\tr_{-\dR\mathrm{Hom}(E_\alpha,E_\alpha)_0} =  Z_\alpha + \frac{\overline{Z}_{\alpha}}{t_1t_2t_3t_4} - \frac{Z_\alpha \overline{Z}_\alpha(1-t_1)(1-t_2)(1-t_3)(1-t_4)}{t_1t_2t_3t_4}.
\end{equation}

\noindent \textbf{Edge term.} Let $\alpha\beta \in E(X)$. On $U_{\alpha\beta} \cong \CC^* \times \mathbb{C}^3$, we use coordinates $x_1,x_2,x_3, x_4$ such that the $(\mathbb{C}^*)^4$-action is given by
$$
t \cdot x_i = t_i x_i, \quad \textrm{for all } i=1,2,3,4 \, \textrm{ and } \, t=(t_1,t_2,t_3,t_4) \in (\CC^*)^4. 
$$
Let $R:=\Gamma(\oO_{U_{\alpha\beta}}) \cong \mathbb{C}[x_1,x_1^{-1}] \otimes \CC[x_2,x_3,x_4]$. As in \cite{MNOP}, we define
$$
\delta(t_1) := \sum_{k \in \ZZ} t_1^k.
$$
Similar to the vertex calculation, we set 
$$
Z_{\alpha\beta} := \tr_{E'|_{U_{\alpha\beta}}}.
$$
In both cases, $E = I_Z$ and $E = I^\mdot$, we have an underlying Cohen-Macaulay support curve $C$ and $C|_{U_\alpha}$, $C|_{U_\beta}$ are described by solid partitions $\pi^{(\alpha)}$, $\pi^{(\beta)}$ respectively. Suppose in both charts $U_\alpha \cong \CC^4$ and $U_\beta \cong \CC^4$, the line $C_{\alpha\beta} \cong \PP^1$ is given by $\{x_2 = x_3 = x_4 = 0\}$. Then $\pi^{(\alpha)}$, $\pi^{(\beta)}$ give rise to the same asymptotic finite plane partition $\lambda^{(\alpha\beta)}$ along the $x_1$-direction, by gluing condition \eqref{glue}, and
\begin{equation} \label{Zlambda}
Z_{\alpha\beta} = \sum_{j,k \geqslant 1} \sum_{l=1}^{\lambda^{(\alpha\beta)}} t_2^{j-1} t_3^{k-1} t_4^{l-1}.
\end{equation}
A computation similar to the vertex case yields
\begin{equation} \label{defEprelim}
- \tr_{-\dR\mathrm{Hom}(E_{\alpha\beta},E_{\alpha\beta})_0} =  \delta(t_1) \left( -Z_{\alpha\beta} + \frac{\overline{Z}_{\alpha\beta}}{t_2t_3t_4} - \frac{Z_{\alpha\beta} \overline{Z}_{\alpha\beta}(1-t_2)(1-t_3)(1-t_4)}{t_2t_3t_4} \right).
\end{equation}

\noindent \textbf{Redistribution.} Expressions \eqref{defVprelim} and \eqref{defEprelim} are formal Laurent series in $t_1,t_2,t_3,t_4$. However, MNOP \cite{MNOP} give a method for redefining these terms in such a way that we get Laurent \emph{polynomials} in $t_1, t_2, t_3, t_4$. This process is known as redistribution. Define
\begin{equation} \label{defF}
\mathsf{F}_{\alpha\beta} := -Z_{\alpha\beta} + \frac{\overline{Z}_{\alpha\beta}}{t_2t_3t_4} - \frac{Z_{\alpha\beta} \overline{Z}_{\alpha\beta}(1-t_2)(1-t_3)(1-t_4)}{t_2t_3t_4}.
\end{equation}
For any $\alpha \in V(X)$, denote by $C_{\alpha\beta_1}, C_{\alpha\beta_2}, C_{\alpha\beta_3}, C_{\alpha\beta_4}$ the four $(\CC^*)^4$-invariant lines passing through the $(\CC^*)^4$-fixed point $p_\alpha$.  Then we define
\begin{equation} \label{defV}
\mathsf{V}_\alpha := \tr_{-\dR\mathrm{Hom}(E_\alpha,E_\alpha)_0} + \sum_{i=1}^{4} \frac{\mathsf{F}_{\alpha\beta_i}(t_{i'},t_{i''},t_{i'''})}{1-t_i},
\end{equation}
where $\{t_i,t_{i'},t_{i''},t_{i'''}\} = \{t_1,t_2,t_3,t_4\}$. For any $\alpha\beta \in E(X)$, use coordinates as in the treatment of the edge term above. Then we define
\begin{equation} \label{defE}
\mathsf{E}_{\alpha\beta} := t_1^{-1} \frac{\mathsf{F}_{\alpha\beta}(t_2,t_3,t_4)}{1-t_1^{-1}} - \frac{\mathsf{F}_{\alpha\beta}(t_2t_1^{-m_{\alpha\beta}},t_3t_1^{-m'_{\alpha\beta}},t_4 t_1^{-m''_{\alpha\beta}})}{1-t_1^{-1}},
\end{equation}
where $(t_1,t_2,t_3,t_4) \mapsto (t_1^{-1}, t_2 t_1^{-m_{\alpha\beta}},  t_3 t_1^{-m'_{\alpha\beta}},  t_4 t_1^{-m''_{\alpha\beta}})$ corresponds to the coordinate transformation $U_\alpha \rightarrow U_\beta$ and $m_{\alpha\beta}, m'_{\alpha\beta}, m''_{\alpha\beta}$ are the weights of the normal bundle of $C_{\alpha\beta}$ defined in \eqref{normal}. The result of this redistibution is 
$$
\mathsf{V}_\alpha \in \ZZ[t_1^{\pm 1}, t_2^{\pm 1}, t_3^{\pm 1}, t_4^{\pm 1}], \qquad \mathsf{E}_{\alpha\beta} \in \ZZ[t_2^{\pm 1}, t_3^{\pm 1}, t_4^{\pm 1}],
$$
which is proved precisely as in \cite[Lem.~9]{MNOP}. When $E = I_Z$, we write $\mathsf{V}_\alpha^{\DT}$ for $\mathsf{V}_\alpha$ and $\mathsf{E}_{\alpha\beta}^{\DT}$ for $\mathsf{E}_{\alpha\beta}$. When $E = I^\mdot$, we write $\mathsf{V}_\alpha^{\PT}$ for $\mathsf{V}_\alpha$ and $\mathsf{E}_{\alpha\beta}^{\PT}$ for $\mathsf{E}_{\alpha\beta}$. Note that $\mathsf{E}_{\alpha\beta}^{\DT} = \mathsf{E}_{\alpha\beta}^{\PT}$. We summarize:
\begin{prop}
Let $X$ be a toric Calabi-Yau 4-fold and $\beta \in H_2(X)$. Let $E = I_Z$, with $Z \in I_n(X,\beta)^{(\CC^*)^4}$ or $E = [I^\mdot] \in P_n(X,\beta)^{(\CC^*)^4}$. In the second case, we assume the Cohen-Macaulay support curve $C$ has the following property: for all $\alpha \in V(X)$, the solid partition associated to $C|_{U_\alpha}$ has at most two non-empty asymptotic plane partitions. Then 
$$
\tr_{-\dR\mathrm{Hom}(E,E)_0} = \sum_{\alpha \in V(X)} \mathsf{V}_\alpha + \sum_{\alpha\beta \in E(X)} \mathsf{E}_{\alpha\beta},
$$
where $\mathsf{V}_\alpha$ and $\mathsf{E}_{\alpha\beta}$ are Laurent polynomials for all $\alpha \in V(X)$ and $\alpha\beta \in E(X)$.
\end{prop}

\subsection{Equivariant $\mathrm{DT/PT}$ correspondence}

Consider plane partitions $\lambda, \mu, \nu, \rho$ of finite size. In the stable pairs case, we assume at most two are non-empty. This data determines a $(\CC^*)^4$-fixed Cohen-Macaulay curve $C \subseteq \CC^4$ with solid partition defined by \eqref{CMsolid}. We denote this solid partition by $\pi_{\mathrm{CM}}(\lambda, \mu, \nu, \rho)$. We are interested in:
\begin{itemize}
\item All $(\CC^*)^4$-invariant closed subschemes $Z \subseteq \CC^4$ with underlying maximal Cohen-Macaulay subcurve $C$. These correspond to solid partitions $\pi$ with asymptotic plane partitions $\lambda, \mu, \nu, \rho$ in directions $1,2,3,4$ (Section \ref{DTfix}). We denote the collection of such solid partitions by $\Pi^{\DT}(\lambda, \mu, \nu, \rho)$. To any solid partition $\pi \in \Pi^{\DT}(\lambda, \mu, \nu, \rho)$, we associate a character $Z_\pi$ defined by RHS of \eqref{Zalpha}. This in turn determines a Laurent polynomial 
$$\mathsf{V}^{\DT}_\pi \in \ZZ[t_1^{\pm 1},t_2^{\pm 1},t_3^{\pm 1},t_4^{\pm 1}]$$ 
by RHS of \eqref{defV} (via \eqref{defF}, \eqref{defVprelim}).
\item All $(\CC^*)^4$-invariant stable pairs $(F,s)$ on $\CC^4$ with underlying Cohen-Macaulay support curve $C$. These correspond to box configurations as described in Section \ref{PTfix}. Denote the collection of such box configurations by $\Pi^{\PT}(\lambda, \mu, \nu, \rho)$. To any box configuration $B \in \Pi^{\PT}(\lambda, \mu, \nu, \rho)$, we associate a character $Z_B$ defined by the RHS of \eqref{PTZalpha}, where the Cohen-Macaulay part is given by \eqref{Zalpha}, for solid partition $\pi_{\mathrm{CM}}(\lambda, \mu, \nu, \rho)$, and the cokernel part is the sum of $t^w$ over all $w \in B$. This determines a Laurent polynomial 
$$\mathsf{V}^{\PT}_\pi  \in \ZZ[t_1^{\pm 1},t_2^{\pm 1},t_3^{\pm 1},t_4^{\pm 1}]$$ 
by RHS of \eqref{defV} (via \eqref{defF}, \eqref{defVprelim}).
\end{itemize} 

From the definitions, one readily calculates (in both the DT/PT case)
\begin{align*}
\overline{\mathsf{V}}_\pi &= \mathsf{V}_\pi \cdot t_1t_2t_3t_4.
\end{align*}
Therefore, after restricting to $\ZZ[t_1,t_2,t_3,t_4] / \langle t_1t_2t_3t_4 - 1 \rangle$, this expression has a square root. This allows us to make the following definition:
\begin{defi}
Consider plane partitions $\lambda, \mu, \nu, \rho$ of finite size. Define the equivariant 4-fold DT vertex by
\footnote{In this definition, we use that $\mathsf{V}_\pi^{\DT}$ does not contain any $T$-fixed terms with coefficient ``$+$'' (no $T$-fixed deformations). This follows from the fact that \eqref{defVprelim} has no $T$-fixed terms with coefficient ``$+$'' (Lemma \ref{Tfixlocus scheme}) and redistribution does not introduce any new $T$-fixed terms. Similar considerations hold in the stable pairs case and for the edge terms discussed below.}
$$
\mathsf{V}_{\lambda\mu\nu\rho}^{\DT}(q)_{o(\lL)} := \sum_{\pi \in \Pi^{\DT}(\lambda, \mu, \nu, \rho)} (-1)^{o(\lL)|_\pi} \sqrt{(-1)^{a_\pi} e_T(-\mathsf{V}^{\DT}_\pi)} \, q^{|\pi|} \in \QQ(\lambda_1,\lambda_2,\lambda_3)(\!(q)\!),
$$
where $o(\lL)|_\pi = 0,1$ gives a choice of sign for each $\pi$, $e_T(-)$ is the equivariant Euler class, and $|\pi|$ denotes renormalized volume. Moreover, $(-1)^{a_\pi}$ denotes the unique sign for which $(-1)^{a_\pi} e_T(-\mathsf{V}^{\DT}_\pi)$ is a square in $\QQ(\lambda_1,\lambda_2,\lambda_3)$. 

Assume at most two of $\lambda, \mu, \nu, \rho$  are non-empty. Define the equivariant 4-fold PT vertex by
$$
\mathsf{V}_{\lambda\mu\nu\rho}^{\PT}(q)_{o(\lL)} := \sum_{B \in \Pi^{\PT}(\lambda, \mu, \nu, \rho)} (-1)^{o(\lL)|_B} \sqrt{(-1)^{a_B} e_T(-\mathsf{V}^{\PT}_B)} \, q^{|B| + |\pi_{\mathrm{CM}}(\lambda, \mu, \nu, \rho)|} \in \QQ(\lambda_1,\lambda_2,\lambda_3)(\!(q)\!),
$$
where $o(\lL)|_B = 0,1$ gives a choice of sign for each $B$, $|B|$ denotes the total number of boxes in the box configuration, and $|\pi_{\mathrm{CM}}(\lambda, \mu, \nu, \rho)|$ denotes renormalized volume. Moreover, $(-1)^{a_B}$ denotes the unique sign for which $(-1)^{a_B} e_T(-\mathsf{V}^{\PT}_B)$ is a square in $\QQ(\lambda_1,\lambda_2,\lambda_3)$. We often omit $o(\lL)$ from the notation. We also define ``normalizations''
\begin{align*}
\widetilde{\mathsf{V}}_{\lambda\mu\nu\rho}^{\DT}(q) :=&\, q^{-|\pi_{\mathrm{CM}}(\lambda, \mu, \nu, \rho)|}\mathsf{V}_{\lambda\mu\nu\rho}^{\DT}(q) \in \QQ(\lambda_1,\lambda_2,\lambda_3)[\![q]\!], \\
\widetilde{\mathsf{V}}_{\lambda\mu\nu\rho}^{\PT}(q) :=&\, q^{-|\pi_{\mathrm{CM}}(\lambda, \mu, \nu, \rho)|}\mathsf{V}_{\lambda\mu\nu\rho}^{\PT}(q) \in \QQ(\lambda_1,\lambda_2,\lambda_3)[\![q]\!].
\end{align*}
\end{defi}

Analogously, we associate to a finite plane partition $\lambda$ the character $Z_\lambda$ defined by the RHS of \eqref{Zlambda}. We define the edge term 
$$\mathsf{E}^{\DT}_{\lambda,o(\lL)} = \mathsf{E}^{\PT}_{\lambda,o(\lL)} \in \QQ(\lambda_1,\lambda_2,\lambda_3)$$ 
by taking the square root of the (signed) Euler class of minus the RHS of \eqref{defE} (via \eqref{defF}). As before, this definition depends on a sign $(-1)^{o(\lL)|_\lambda}$, where $o(\lL)|_\lambda=0,1$. We usually omit this dependence from the notation.

The vertex formalism reduces the calculation of equivariant curve counting and stable pair invariants of any toric Calabi-Yau 4-fold to a combinatorial expression involving $\mathsf{V}_{\lambda\mu\nu\rho}$ and $\mathsf{E}_{\lambda}$. Rather than writing this down in general, we discuss an illustrative example. Suppose $X$ is the total space of $\oO_{\PP^2}(-1) \oplus \oO_{\PP^2}(-2)$. Let $\beta = d\,[\PP^1]$. Then Lemma \ref{chi} and the vertex formalism (Section \ref{vertex section}) imply
\begin{align*}
&I_{\beta}(X;T)=\sum_{\lambda,\mu,\nu \atop |\lambda|+|\mu|+|\nu| = d} q^{f_{1,-1,-2}(\lambda)+f_{1,-1,-2}(\mu)+f_{1,-1,-2}(\nu)} \mathsf{E}^{\DT}_{\lambda}|_{(\lambda_1,\lambda_2,\lambda_3,\lambda_4)} \mathsf{V}_{\lambda\mu\varnothing\varnothing}^{\DT}|_{(\lambda_1,\lambda_2,\lambda_3,\lambda_4)} \mathsf{E}^{\DT}_{\mu} |_{(\lambda_2,\lambda_1,\lambda_3,\lambda_4)}  \\
&\cdot \mathsf{V}^{\DT}_{\mu\nu\varnothing\varnothing} |_{(-\lambda_2,\lambda_1-\lambda_2,\lambda_3+\lambda_2,\lambda_4+2\lambda_2)} \mathsf{E}^{\DT}_{\nu}|_{(\lambda_1 - \lambda_2, -\lambda_2,\lambda_3+\lambda_2,\lambda_4+2\lambda_2)} \mathsf{V}^{\DT}_{\nu\lambda\varnothing\varnothing}|_{(\lambda_2-\lambda_1,-\lambda_1,\lambda_3+\lambda_1,\lambda_4+2\lambda_1)},
\end{align*}
where the sum is over all finite plane partitions $\lambda, \mu, \nu$ satisfying $|\lambda|+|\mu|+|\nu| = d$. Here the choice of signs for $I_{\beta}(X;T)$ is determined by the choice of signs in each vertex and edge term. The same expression holds for $P_{\beta}(X;T)$ replacing $\DT$ by $\PT$.

When $\beta=0$, the space $I_{n}(X,\beta)$ reduces to the Hilbert scheme $\Hilb^n(X)$ of $n$ points previously studied in \cite{CK, Nekrasov, NP}. 
When $X = \CC^4$, we have $I_0(X;T) = \mathsf{V}_{\varnothing\varnothing\varnothing\varnothing}^{\DT}(q)$. A closed expression for this generating series was conjectured by Nekrasov \cite{Nekrasov}:
\begin{conj}\label{nekrasov conj}$(\mathrm{Nekrasov}$ \cite{Nekrasov}$\mathrm{)}$
There exist unique choices of signs such that the following formula holds
\begin{align*}
\mathsf{V}_{\varnothing\varnothing\varnothing\varnothing}^{\DT}(q) =\exp\bigg(q \frac{(\lambda_1+\lambda_2)(\lambda_1+\lambda_3)(\lambda_2+\lambda_3)}{\lambda_1\lambda_2\lambda_3(\lambda_1+\lambda_2+\lambda_3)}\bigg). \end{align*}
\end{conj}
\begin{rmk}
The existence part of the conjecture was verified modulo $q^7$ in \cite{Nekrasov} and later modulo $q^{17}$ in \cite{NP}. 
The uniqueness was conjectured in \cite[Appendix B]{CK}  and checked modulo $q^5$. Note that
$$
\frac{(\lambda_1+\lambda_2)(\lambda_1+\lambda_3)(\lambda_2+\lambda_3)}{\lambda_1\lambda_2\lambda_3(\lambda_1+\lambda_2+\lambda_3)} =-\int_{\CC^4}c^T_3(\CC^4),
$$
where $\int_{\CC^4}$ is equivariant push-forward to a point and $c^T_3(\CC^4)$ denotes the equivariant third Chern class of the tangent bundle $T_{\CC^4}$. 
\end{rmk}

Let $X$ be any toric Calabi-Yau 4-fold. The vertex formalism (for $\beta=0$), Atiyah-Bott localization on $X$, and Nekrasov's conjecture at once imply the following result about $\Hilb^n(X)$:
\begin{prop}
Assuming Conjecture \ref{nekrasov conj} holds, then for any toric Calabi-Yau 4-fold $X$, there exist choices of signs such that 
\begin{align*} I_{0}(X;T)=\exp\bigg(-q \int_Xc_3^T(X)\bigg). \end{align*}
\end{prop}

We conjecture the following DT/PT vertex correspondence:
\begin{conj} \label{affine DT/PT conj}
For any finite plane partitions $\lambda, \mu, \nu, \rho$, at most two of which are non-empty, there are choices of signs such that
$$
\mathsf{V}_{\lambda\mu\nu\rho}^{\DT}(q) = \mathsf{V}_{\lambda\mu\nu\rho}^{\PT}(q) \, \mathsf{V}_{\varnothing\varnothing\varnothing\varnothing}^{\DT}(q).
$$
Suppose we choose the signs for $\mathsf{V}_{\varnothing\varnothing\varnothing\varnothing}^{\DT}(q)$ equal to the unique signs in Nekrasov's conjecture \ref{nekrasov conj}. Then, at each order in $q$, the choice of signs for which LHS and RHS agree is unique up to an overall sign.
\end{conj}
By using an implementation into Maple, we verified the following cases:
\begin{prop}\label{evidencethm} 
There are choices of signs such that 
$$\widetilde{\mathsf{V}}_{\lambda\mu\nu\rho}^{\DT}(q) = \widetilde{\mathsf{V}}_{\lambda\mu\nu\rho}^{\PT}(q) \, \mathsf{V}_{\varnothing\varnothing\varnothing\varnothing}^{\DT}(q) \mod q^N
$$
in the following cases:
\begin{itemize}
\item for any $|\lambda| + |\mu| + |\nu| + |\rho| \leqslant 1$ and $N=4$,
\item for any $|\lambda| + |\mu| + |\nu| + |\rho| \leqslant 2$ and $N=4$,
\item for any $|\lambda| + |\mu| + |\nu| + |\rho| \leqslant 3$ and $N=3$,
\item for any $|\lambda| + |\mu| + |\nu| + |\rho| \leqslant 4$ and $N=3$.
\end{itemize}
In each of these cases, the uniqueness statement of Conjecture \ref{affine DT/PT conj} holds.
\end{prop}

\subsection{Primary insertions}

Let $X$ be a toric Calabi-Yau 4-fold and $\gamma \in H_T^*(X,\QQ)$. Consider $[E] \in M^{(\CC^*)^4}$, where either $([E],M) = (Z, I_n(X,\beta))$ or $([E],M) = ([I^\mdot = \{\oO_X \rightarrow F\}], P_n(X,\beta))$. Denote the underlying Cohen-Macaulay curve by $C$. We are interested in the restriction
$$
\tau(\gamma) \big|_{[E]}\,\,.
$$
Going around the following diagram both ways
\begin{displaymath}
\xymatrix
{
\{[E]\} \times X \ar_{\int_X}[d] \ar@{^(->}[r] & M \times X \ar^{\pi_M}[d] \ar^{\quad \pi_X}[r] & X  \\
\bullet \ar[r] & M &
}
\end{displaymath}
we find
\begin{equation} \label{tau1}
\tau(\gamma) \big|_{[E]} = \int_X \gamma \cdot \ch^T_3(\oO_C),
\end{equation}
where $\int_X$ denotes equivariant push-forward to a point $\bullet$. Crucially, we used
\begin{equation} \label{chmatch}
\ch^T_3(\oO_Z) = \ch^T_3(\oO_C) = \ch^T_3(F)
\end{equation}
so $\tau(\gamma) |_{[E]}$ is given by the same expression in the DT and PT case. This is no longer the case for higher descendent insertions, which are not considered in this paper. Applying Atiyah-Bott localization to \eqref{tau1} gives
\begin{equation} \label{taufinal}
\tau(\gamma)  \big|_{[E]} = \sum_{\alpha \in V(X)} \ch^T_3( Z_\alpha \cdot (1-t^{(\alpha)}_1)(1-t^{(\alpha)}_2)(1-t^{(\alpha)}_3)(1-t^{(\alpha)}_4)) \cdot \frac{\gamma\big|_{p_\alpha}}{e_T(T_{X}|_{p_\alpha})},
\end{equation}
where $t_i^{(\alpha)}$ are the characters of the $(\CC^*)^4$-action on chart $U_\alpha$ and $Z_\alpha$ is given by \eqref{Zalpha} with $\pi^{(\alpha)}$ the solid partition corresponding to $C|_{U_\alpha}$. Insertions for stable pair invariants were studied on 3-folds in \cite[Sect.~6]{PT2}.

Similar to the previous paragraph, the vertex formalism can also be used to calculate arbitrary  equivariant curve counting and stable pairs invariants with primary insertions on toric Calabi-Yau 4-folds. Again we illustrate this for $X$ the total space of $\oO_{\PP^2}(-1) \oplus \oO_{\PP^2}(-2)$ and $\beta = d\,[\PP^1]$. For any point-like plane partitions $\lambda, \mu, \nu$, denote by $C_{\lambda\mu\nu}$ the $(\CC^*)^4$-invariant Cohen-Macaulay curve supported on the $(\CC^*)^4$-invariant lines $\PP^1 \cup \PP^1 \cup \PP^1 \subseteq \PP^2 \subseteq X$ and with ``cross-sections'' $\lambda, \mu, \nu$ along these lines. We denote its character in chart $U_\alpha$ by $(C_{\lambda\mu\nu})_\alpha$. Then Lemma \ref{chi}, the vertex formalism (Section \ref{vertex section}), and \eqref{taufinal} imply 
\begin{align*}
&I_{\beta}(X;T)(\gamma_1, \ldots, \gamma_m)=\sum_{\lambda,\mu,\nu \atop |\lambda|+|\mu|+|\nu| = d} q^{f_{1,-1,-2}(\lambda)+f_{1,-1,-2}(\mu)+f_{1,-1,-2}(\nu)} \mathsf{E}^{\DT}_{\lambda}|_{(\lambda_1,\lambda_2,\lambda_3,\lambda_4)} \mathsf{V}_{\lambda\mu\varnothing\varnothing}^{\DT}|_{(\lambda_1,\lambda_2,\lambda_3,\lambda_4)} \\
&\cdot \mathsf{E}^{\DT}_{\mu} |_{(\lambda_2,\lambda_1,\lambda_3,\lambda_4)} \mathsf{V}^{\DT}_{\mu\nu\varnothing\varnothing} |_{(-\lambda_2,\lambda_1-\lambda_2,\lambda_3+\lambda_2,\lambda_4+2\lambda_2)} \mathsf{E}^{\DT}_{\nu}|_{(\lambda_1 - \lambda_2, -\lambda_2,\lambda_3+\lambda_2,\lambda_4+2\lambda_2)} \mathsf{V}^{\DT}_{\nu\lambda\varnothing\varnothing}|_{(\lambda_2-\lambda_1,-\lambda_1,\lambda_3+\lambda_1,\lambda_4+2\lambda_1)} \\
&\cdot \Bigg(  \prod_{i=1}^{m} \sum_{\alpha \in V(X)} \ch^T_3\big((C_{\lambda\mu\nu})_\alpha (1-t_1^{(\alpha)})(1-t_2^{(\alpha)})(1-t_3^{(\alpha)})(1-t_4^{(\alpha)}) \big) \frac{\gamma_i \big|_{p_\alpha}}{e_T(T_{X}|_{p_{\alpha}})} \Bigg), 
\end{align*}
for all $\gamma_1, \ldots, \gamma_m \in H_T^*(X,\QQ)$. The choice of signs for $I_\beta(X;T)(\gamma_1, \ldots, \gamma_m)$ are determined by RHS. The same expression holds for $P_{\beta}(X;T)(\gamma_1, \ldots, \gamma_m)$ when replacing $\DT$ by $\PT$.

\subsection{Consequences}

The DT/PT vertex correspondence (Conjecture \ref{affine DT/PT conj}) implies the following equivariant DT/PT correspondence for toric Calabi-Yau 4-folds.
\begin{thm}\label{affine implies toric}
Assume Conjecture \ref{affine DT/PT conj} holds. Let $X$ be a toric Calabi-Yau 4-fold, $\beta \in H_2(X)$ such that $P_n(X,\beta)^{(\CC^*)^4}$ is at most 0-dimensional for all $n \in \ZZ$, and let
$\gamma_1, \ldots, \gamma_m \in H^*_T(X)$. Then there exist choices of signs such that
\begin{align*}
\frac{I_{\beta}(X;T)(\gamma_1, \ldots, \gamma_m)}{I_{0}(X;T)}=P_{\beta}(X;T)(\gamma_1, \ldots, \gamma_m).
\end{align*}
In particular, 
without insertions we have
\begin{align*}
\frac{I_{\beta}(X;T)}{I_{0}(X;T)}= P_{\beta}(X;T).
\end{align*}
\end{thm}
\begin{proof}
In order to keep the notation sufficiently simple, we consider the case $X$ is the total space of $\mathrm{Tot}_{\mathbb{P}^2}(\oO(-1) \oplus \oO(-2))$ and $\beta=d\,[\PP^1]$. The general case follows similarly, after setting up the right notation. As in the previous paragraph, for any finite plane partitions $\lambda, \mu, \nu$, let $C_{\lambda\mu\nu}$ be the $(\CC^*)^4$-invariant Cohen-Macaulay curve supported on the $(\CC^*)^4$-invariant lines $\PP^1 \cup \PP^1 \cup \PP^1 \subseteq \PP^2 \subseteq X$ and with ``cross-sections'' $\lambda, \mu, \nu$ along these lines. We denote its character in chart $U_\alpha$ by $(C_{\lambda\mu\nu})_\alpha$. Then Conjecture \ref{affine DT/PT conj} implies that there exist choices of signs such that
\begin{align*}
&I_{\beta}(X;T)(\gamma_1, \ldots, \gamma_m)=\sum_{\lambda,\mu,\nu \atop |\lambda|+|\mu|+|\nu| = d} q^{f_{1,-1,-2}(\lambda)+f_{1,-1,-2}(\mu)+f_{1,-1,-2}(\nu)} \mathsf{E}^{\DT}_{\lambda}|_{(\lambda_1,\lambda_2,\lambda_3,\lambda_4)} \mathsf{V}_{\lambda\mu\varnothing\varnothing}^{\DT}|_{(\lambda_1,\lambda_2,\lambda_3,\lambda_4)} \\
&\cdot \mathsf{E}^{\DT}_{\mu} |_{(\lambda_2,\lambda_1,\lambda_3,\lambda_4)} \mathsf{V}^{\DT}_{\mu\nu\varnothing\varnothing} |_{(-\lambda_2,\lambda_1-\lambda_2,\lambda_3+\lambda_2,\lambda_4+2\lambda_2)} \mathsf{E}^{\DT}_{\nu}|_{(\lambda_1 - \lambda_2, -\lambda_2,\lambda_3+\lambda_2,\lambda_4+2\lambda_2)} \mathsf{V}^{\DT}_{\nu\lambda\varnothing\varnothing}|_{(\lambda_2-\lambda_1,-\lambda_1,\lambda_3+\lambda_1,\lambda_4+2\lambda_1)} \\
&\cdot \Bigg(  \prod_{i=1}^{m} \sum_{\alpha \in V(X)} \ch^T_3\big((C_{\lambda\mu\nu})_\alpha (1-t_1^{(\alpha)})(1-t_2^{(\alpha)})(1-t_3^{(\alpha)})(1-t_4^{(\alpha)})\big) \frac{\gamma_i \big|_{p_\alpha}}{e_T(T_{X}|_{p_{\alpha}})} \Bigg), \\
&= \mathsf{V}_{\varnothing\varnothing\varnothing\varnothing}^{\DT}\Big|_{(\lambda_1,\lambda_2,\lambda_3,\lambda_4)}  \mathsf{V}^{\DT}_{\varnothing\varnothing\varnothing\varnothing}|_{(-\lambda_2,\lambda_1-\lambda_2,\lambda_3+\lambda_2,\lambda_4+2\lambda_2)}  \mathsf{V}^{\DT}_{\varnothing\varnothing\varnothing\varnothing} |_{(\lambda_2-\lambda_1,-\lambda_1,\lambda_3+\lambda_1,\lambda_4+2\lambda_1)}  \\
&\cdot \sum_{\lambda,\mu,\nu \atop |\lambda|+|\mu|+|\nu| = d} q^{f_{1,-1,-2}(\lambda)+f_{1,-1,-2}(\mu)+f_{1,-1,-2}(\nu)} \mathsf{E}^{\PT}_{\lambda}|_{(\lambda_1,\lambda_2,\lambda_3,\lambda_4)} \mathsf{V}_{\lambda\mu\varnothing\varnothing}^{\PT}|_{(\lambda_1,\lambda_2,\lambda_3,\lambda_4)} \\
&\cdot \mathsf{E}^{\PT}_{\mu} |_{(\lambda_2,\lambda_1,\lambda_3,\lambda_4)} \mathsf{V}^{\PT}_{\mu\nu\varnothing\varnothing} |_{(-\lambda_2,\lambda_1-\lambda_2,\lambda_3+\lambda_2,\lambda_4+2\lambda_2)} \mathsf{E}^{\PT}_{\nu}|_{(\lambda_1 - \lambda_2, -\lambda_2,\lambda_3+\lambda_2,\lambda_4+2\lambda_2)} \mathsf{V}^{\PT}_{\nu\lambda\varnothing\varnothing}|_{(\lambda_2-\lambda_1,-\lambda_1,\lambda_3+\lambda_1,\lambda_4+2\lambda_1)} \\ 
&\cdot \Bigg(  \prod_{i=1}^{m} \sum_{\alpha \in V(X)} \ch^T_3\big((C_{\lambda\mu\nu})_\alpha (1-t_1^{(\alpha)})(1-t_2^{(\alpha)})(1-t_3^{(\alpha)})(1-t_4^{(\alpha)})\big) \frac{\gamma_i \big|_{p_\alpha}}{e_T(T_{X}|_{p_{\alpha}})} \Bigg), \\
&= I_{0}(X;T)\cdot P_{\beta}(X;T)(\gamma_1, \ldots, \gamma_m).
\end{align*}
Besides Conjecture \ref{affine DT/PT conj}, this uses the fact that $\mathsf{E}^{\DT}_{\lambda} = \mathsf{E}^{\PT}_{\lambda}$ (the DT/PT edges coincide) and $\tau(\gamma)|_{Z} = \tau(\gamma)|_{[I^{\mdot}]}$ for any $Z \in I_n(X,\beta)^{(\CC^*)^4}$ and $[I^\mdot] \in P_n(X,\beta)^{(\CC^*)^4}$ with the same underlying Cohen-Macaulay curve \eqref{chmatch}. 
\end{proof}

The next two theorems can be seen as good evidence for Conjecture \ref{conj DT/PT}. 
\begin{thm} \label{mainthm1}
Let $X$ be a toric Calabi-Yau 4-fold and $\beta \in H_2(X)$. Suppose $P_n(X, \beta)$ is proper and $P_n(X, \beta)^{(\CC^*)^4}$ is at most 0-dimensional for all $n \in \ZZ$. Assume the following:
\begin{enumerate} 
\item the DT/PT vertex correspondence (Conjecture \ref{affine DT/PT conj}) holds,
\item \eqref{virloc1} holds for $\beta$ and $n=0$, and \eqref{virloc2} holds for $\beta$ and all $n <0$, 
\item the signs of (1) and (2) can be chosen compatibly. 
\end{enumerate}
Then
$
I_{0,\beta}(X;T)=P_{0,\beta}(X) \in \ZZ.
$
\end{thm}
\begin{proof}
Theorem \ref{affine implies toric} implies
\begin{equation*}
I_{0,\beta}(X;T)=\sum_{k=0}^{\infty}I_{k,0}(X;T)\cdot 
P_{-k,\beta}(X;T),
\end{equation*}
which is a finite sum because $P_{-k}(X,\beta)^{(\CC^*)^4} = \varnothing$ for $k \gg 0$. 
For all $k>0$, 
\eqref{virloc2} implies
$$P_{-k,\beta}(X;T)=0. $$
We conclude $I_{0,\beta}(X;T)=P_{0,\beta}(X;T) = P_{0,\beta}(X) \in \ZZ$, where the second equality follows from the localization formula \eqref{virloc1}.
\end{proof}

\begin{thm} \label{mainthm2}
Let $X$ be a toric Calabi-Yau 4-fold satisfying $\int_X c_3^T(X) = 0$ and let $\beta \in H_2(X)$. Let $\gamma_1, \ldots, \gamma_m \in H^*(X,\ZZ)$ admitting $T$-equivariant lifts and satisfying $\sum_i \deg \tau(\gamma_i) = 2n > 0$. Suppose $P_\chi(X, \beta)$ is proper and $P_\chi(X, \beta)^{(\CC^*)^4}$ is at most 0-dimensional for all $\chi \in \ZZ$. Assume the following:
\begin{enumerate} 
\item the DT/PT vertex correspondence (Conjecture \ref{affine DT/PT conj}) holds,
\item Nekrasov's conjecture (Conjecture \ref{nekrasov conj}) holds,
\item \eqref{virloc1} holds for $\beta$, $\gamma_1, \ldots, \gamma_m$, $n$, 
\item the signs of (1), (2), and (3) can be chosen compatibly. 
\end{enumerate}
Then 
$
I_{n,\beta}(X;T)(\gamma_1, \ldots, \gamma_m)=P_{n,\beta}(X)(\gamma_1, \ldots, \gamma_m) \in \ZZ.
$
\end{thm}

\begin{proof}
Choose insertions $\gamma_1, \ldots, \gamma_m \in H^*(X,\ZZ)$ (admitting $T$-equivariant lifts) with the property that $\sum_i \deg \tau(\gamma_i) = 2n > 0$. Theorem \ref{affine implies toric} implies
\begin{equation} \label{eqn1}
I_{n,\beta}(X;T)(\gamma_1, \ldots, \gamma_m)=\sum_{k=0}^{\infty}I_{k,0}(X;T)\cdot 
P_{n-k,\beta}(X;T)(\gamma_1, \ldots, \gamma_m),
\end{equation}
which is a finite sum because $P_{-k}(X,\beta)^{(\CC^*)^4} = \varnothing$ for $k \gg 0$. Nekrasov's conjecture and $\int_X c_3^T(X) = 0$ imply
$$
I_0(X;T) = \exp\left(-q \int_X c_3^T(X) \right) = 1.
$$
Hence $I_{k,0}(X;T) = 0$ for all $k>0$ and \eqref{eqn1} reduces to
\begin{equation*}
I_{n,\beta}(X;T)(\gamma_1, \ldots, \gamma_m)=P_{n,\beta}(X;T)(\gamma_1, \ldots, \gamma_m) = P_{n,\beta}(X)(\gamma_1, \ldots, \gamma_m) \in \ZZ,
\end{equation*}
where the second equality follows from \eqref{virloc1}.
\end{proof}
\begin{rmk}\label{rmk on remove c3}
Upon taking an appropriate equivariant limit of the equivariant parameters, we expect that the virtual localization formula \eqref{virloc1} generalizes to the case where the degree of the virtual class does not match with the degree of the insertion. Then the only non-zero term in \eqref{eqn1} is the one where the degrees match, and the assumption $\int_X c_3^T(X) = 0$ in the theorem can be removed.
\end{rmk}

\subsection{Examples} \label{examplessec}

We now discuss some cases of Theorems \ref{affine implies toric}, \ref{mainthm1}, and \ref{mainthm2}, where we verified Conjectures \ref{nekrasov conj} and \ref{affine DT/PT conj} (by Proposition \ref{evidencethm}). In each case the verification is for curve classes $\beta$ of total degree $\leqslant 4$ and holomorphic Euler characteristic $n$ with $n \leqslant N$ for some $N$. The cases $n=0,1$ are particularly interesting cases in view of the PT/GV correspondence of \cite{CMT2} discussed in Section \ref{PT/GV section}.

Let $X = \mathrm{Tot}_{\PP^1}(\oO \oplus \oO(-1) \oplus \oO(-1))$. Then (the proof of) Theorem \ref{affine implies toric} implies that there exist choices of signs such that
\begin{equation} \label{examples sec eqn}
I_{\beta}(X;T)(\gamma_1, \ldots, \gamma_m)=I_{0}(X;T)\cdot P_{\beta}(X;T)(\gamma_1, \ldots, \gamma_m) \mod q^N
\end{equation}
in the following cases: 
\begin{itemize}
\item $\beta=1$ and $N=5$,
\item $\beta=2$ and $N=6$,
\item $\beta=3$ and $N=6$,
\item $\beta=4$ and $N=7$.
\end{itemize}

Let $X = \mathrm{Tot}_{\PP^2}(\oO(-1) \oplus \oO(-2))$. Then (the proof of) Theorem \ref{affine implies toric} implies that there exist choices of signs such that \eqref{examples sec eqn} holds in the following cases: 
\begin{itemize}
\item $\beta=1$ and $N=5$,
\item $\beta=2$ and $N=5$,
\item $\beta=3$ and $N=3$,
\item $\beta=4$ and $N=1$.
\end{itemize}
Note that $P_n(X, \beta)$ is proper for all $n \in \ZZ$, because $\oO(-1) \oplus \oO(-2)$ has two negative fibre directions. Assuming the virtual localization formulae \eqref{virloc1} (for $\beta=3,4$ and $n=0$), \eqref{virloc2} (for $\beta=3,4$ and $n<0$), and compatibility of signs, we also conclude that  $I_{0,\beta}(X;T)=P_{0,\beta}(X) \in \ZZ$ for $\beta=3,4$. In App.~\ref{app}, we prove the virtual localization formula when all stable pairs are scheme theoretically supported in the zero section and the moduli space is smooth, which holds for $\beta=3,4$, $n=0$.

Let $X = \mathrm{Tot}_{\PP^1 \times \PP^1}(\oO(-1,-1) \oplus \oO(-1,-1))$. Then (the proof of) Theorem \ref{affine implies toric} implies that there exist choices of signs such that \eqref{examples sec eqn} holds in the following cases:
 \begin{itemize}
\item $\beta=(1,0)$ and $N=5$,
\item $\beta=(2,0)$ and $N=6$,
\item $\beta=(1,1)$ and $N=5$,
\item $\beta=(3,0)$ and $N=6$,
\item $\beta=(2,1)$ and $N=4$,
\item $\beta=(4,0)$ and $N=7$,
\item $\beta=(3,1)$ and $N=4$,
\item $\beta=(2,2)$ and $N=3$.
\end{itemize}
Again, $P_n(X, \beta)$ is proper for all $n \in \ZZ$, because $\oO(-1,-1) \oplus \oO(-1,-1)$ has two negative fibre directions. Also note that $\int_X c_3^T(X) = 0$ (by direct calculation). Assuming the virtual localization formulae \eqref{virloc1} (for $\beta=(2,2)$ and $n=0$), \eqref{virloc2} (for $\beta=(2,2)$ and $n<0$), and compatibility of signs, we also conclude that  $I_{0,(2,2)}(X;T)=P_{0,(2,2)}(X) \in \ZZ$. For $\beta=(2,2)$ and $n=0$, all stable pairs are scheme theoretically supported in the zero section and the moduli space is smooth, so the virtual localization formula holds by App.~\ref{app}. 

Let $\beta$, $n>0$ be in the previous list. Let $\gamma_1, \ldots, \gamma_m \in H^*(X,\ZZ)$ admitting $T$-equivariant lifts and satisfying $\sum_i \deg \tau(\gamma_i) = 2n>0$. In this range, Nekrasov's conjecture (Conjecture \ref{nekrasov conj}) holds. Assuming \eqref{virloc1} (for $\beta, \gamma_1, \ldots, \gamma_m, n$) and compatibility of signs, we conclude $I_{n,\beta}(X;T)(\gamma_1, \ldots, \gamma_m)=P_{n,\beta}(X)(\gamma_1, \ldots, \gamma_m) \in \ZZ$. Moreover, in the following cases all stable pairs are scheme theoretically supported in the zero section, so the virtual localization formula is established in Appendix \ref{app}:
\begin{itemize}
\item $\beta=(1,0)$ and any $n$,
\item $\beta=(2,0)$ and $n \leqslant 2$,
\item $\beta=(1,1)$ and any $n$,
\item $\beta=(3,0)$ and $n \leqslant 3$,
\item $\beta=(2,1)$ and $n \leqslant 2$,
\item $\beta=(4,0)$ and $n \leqslant 4$,
\item $\beta=(3,1)$ and $n \leqslant 2$,
\item $\beta=(2,2)$ and $n \leqslant 2$.
\end{itemize}

\subsection{DT/PT generating series of a local curve}

Consider the toric 4-fold 
$$
X=\mathrm{Tot}_{\mathbb{P}^1}(\oO(l_1) \oplus \oO(l_2) \oplus \oO(l_3)).
$$ 
Suppose $l_1+l_2+l_3=-2$, then $X$ is Calabi-Yau. There are two maximal $(\CC^*)^4$-invariant affine open subsets and the transition map is given by
\begin{align*}(x_1,x_2,x_3,x_4)\mapsto  (x_1^{-1},x_2x_1^{-l_1},x_3x_1^{-l_2},x_4x_1^{-l_3}), \end{align*}
where
\begin{align*}t\cdot (x_1,x_2,x_3,x_4)=(t_1x_1,t_2x_2,t_3x_3,t_4x_4) \qquad \textrm{for all } t=(t_1,t_2,t_3,t_4) \in (\CC^*)^4. \end{align*}
Using the identification $H_2(X)\cong \mathbb{Z}$ (where the class of the zero section $[\PP^1]$ corresponds to $1$), we have (equivariant) stable pair invariants 
\begin{align*}P_{n,d}(X; T)\in \frac{\mathbb{Q}(\lambda_1,\lambda_2,\lambda_3,\lambda_4)}{\langle\lambda_1+\lambda_2+\lambda_3+\lambda_4\rangle} \cong \QQ(\lambda_1,\lambda_2,\lambda_3), \end{align*}
for all $n, d \in \ZZ$.
Motivated by \cite{CMT2}, we propose the following:
\begin{conj}\label{pair inv gene series}
For $X=\mathrm{Tot}_{\PP^1}(\oO \oplus \oO(-1) \oplus \oO(-1))$, there exist choices of signs such that the following equation holds
\begin{align*}\sum_{n,d\geqslant0}P_{n,d}(X; T)\,q^n y^d=\exp\Big(\frac{q y}{\lambda_2}\Big), \end{align*}
where $\lambda_2$ is the equivariant parameter for the $\CC^*$-action on the first fibre $\oO_{\mathbb{P}^1}$ and $P_{0,0}(X; T):=1$.
\end{conj}
\begin{rmk} \label{reformulate}
The above conjecture is equivalent to 
\begin{align*}P_{n,n}(X; T)=\frac{1}{n!\, \lambda_2^n}, \quad P_{n,d}(X; T)=0, \,\, \mathrm{if}\, n\neq d, \end{align*}
for all $n,d \geqslant 0$. This can be viewed as an equivariant analogue of Conjecture \ref{pair/GW g=0 conj intro}.
\end{rmk}
Using the vertex formalism and a Maple implementation, we obtain the following:
\begin{prop}\label{verify pair inv gene series}
Conjecture \ref{pair inv gene series} (formulated as in Rem.~\ref{reformulate}) is true in the following cases:
\begin{itemize}
\item for any $n \leqslant d$,
\item $d=1$ and modulo $q^6$,
\item $d=2$ and modulo $q^6$,
\item $d=3$ and modulo $q^6$,
\item $d=4$ and modulo $q^7$.
\end{itemize}
\end{prop}
\begin{proof}
By Lemma \ref{chi}, $P_n(X,d)^{(\CC^*)^4}$ is empty unless $ n \geqslant d$. For $n=d$, the moduli space $P_n(X,d)^{(\CC^*)^4}$ only contains one element, namely the Cohen-Macaulay curve obtained by the $n$ times thickening of the zero section into the $\oO_{\PP^1}$-direction. Equation \eqref{defF} gives
\begin{align*}
\mathsf{F} =&\,\frac{t_{1}^{-1}}{1-t_1^{-1}} \Bigg[ -\sum_{i=0}^{n-1} t_2^{i} + \frac{1}{t_2t_3t_4} \sum_{i=0}^{n-1} t_2^{-i} - \frac{(1-t_2)(1-t_3)(1-t_4)}{t_2t_3t_4} \sum_{i=0}^{n-1} t_2^{i} \sum_{j=0}^{n-1} t_2^{-j}  \Bigg] \\
&\,- \frac{1}{1-t_1^{-1}} \Bigg[ -\sum_{i=0}^{n-1} t_2^{i} + \frac{1}{t_1^2 t_2t_3t_4} \sum_{i=0}^{n-1} t_2^{-i} - \frac{(1-t_2)(1-t_3t_1)(1-t_4t_1)}{t_1^2t_2t_3t_4} \sum_{i=0}^{n-1} t_2^{i} \sum_{j=0}^{n-1} t_2^{-j}  \Bigg] \\
=&\, \sum_{i=0}^{n-1} t_2^{i} + \sum_{i=0}^{n-1} t_2^{-i} \\
&+\Bigg\{ - \frac{t_{1}^{-1}}{1-t_1^{-1}} \frac{(1-t_2)(1-t_3)(1-t_4)}{t_2t_3t_4} + \frac{1}{1-t_1^{-1}}  \frac{(1-t_2)(1-t_3t_1)(1-t_4t_1)}{t_1^2t_2t_3t_4} \Bigg\}  \frac{1-t_2^n}{1-t_2} \frac{1-t_2^{-n}}{1-t_2^{-1}} \\
=&\, \sum_{i=0}^{n-1} t_2^{i} + \sum_{i=0}^{n-1} t_2^{-i} - (1-t_2)(1-t_2^{-1})  \frac{1-t_2^n}{1-t_2} \frac{1-t_2^{-n}}{1-t_2^{-1}} = \sum_{i=1}^{n} (t_2^{i} + t_2^{-i}),
\end{align*}
where we used the relation $t_1t_2t_2t_4 = 1$ in the second and third equality. Taking the square root of the (signed) Euler class of minus this expression gives $\pm 1/(n! \lambda_2^n)$, as desired.

The other cases were checked using the vertex formalism and a Maple implementation.
\end{proof}

As an application, we calculate all curve counting invariants of $\mathrm{Tot}_{\PP^1}(\oO \oplus \oO(-1) \oplus \oO(-1))$. The vertex formalism (Section \ref{vertex section}) at once implies:
\begin{cor}
Let $X=\mathrm{Tot}_{\PP^1}(\oO \oplus \oO(-1) \oplus \oO(-1))$. Assume the following:
\begin{enumerate}
\item the DT/PT vertex correspondence holds (Conjecture \ref{affine DT/PT conj}),
\item Nekrasov's conjecture holds (Conjecture \ref{nekrasov conj}),
\item the local curve conjecture holds (Conjecture \ref{pair inv gene series}),
\item the choices of signs of (1)--(3) are compatible.
\end{enumerate}
Then there exist choices of signs such that 
\begin{align*}\sum_{n,d\geqslant0}I_{n,d}(X; T)q^n y^d=
\exp\bigg(\frac{q}{\lambda_2}\bigg(y+\frac{(\lambda_1+\lambda_2)(\lambda_1+\lambda_3)(\lambda_2+\lambda_3)}{\lambda_1\lambda_3(\lambda_1+\lambda_2+\lambda_3)}+\frac{\lambda_3(\lambda_1-\lambda_2)(\lambda_1+\lambda_2+\lambda_3)}{\lambda_1(\lambda_1+\lambda_3)(\lambda_2+\lambda_3)}\bigg)\bigg). \end{align*}
\end{cor}

\appendix

\section{Virtual localization and relative Hilbert schemes} \label{app}

Consider the local surfaces $X = \mathrm{Tot}_{\PP^2}(\oO(-1) \oplus \oO(-2))$, $\mathrm{Tot}_{\PP^1 \times \PP^1}(\oO(-1,-1) \oplus \oO(-1,-1))$ and denote the zero section by $S \subseteq X$. For curve classes of small degree, it is shown in \cite{CKM} (similar to \cite[Sect. 3.2]{CMT2}) that all stable pairs are scheme theoretically supported on $S \subseteq X$ and $P_n(X,\beta) \cong P_n(S,\beta)$ is smooth. In this case, stable pair invariants can be computed 
using relative Hilbert schemes.

\begin{thm}\label{check loc formula}
Let $X = \mathrm{Tot}_{\PP^2}(\oO(-1) \oplus \oO(-2))$ or $\mathrm{Tot}_{\PP^1 \times \PP^1}(\oO(-1,-1) \oplus \oO(-1,-1))$ and denote the zero section by $S \subseteq X$. Suppose $\beta \in H_2(X)$ and $n \in \ZZ$ are chosen such that $P_n(X,\beta) \cong P_n(S,\beta)$ is smooth. Then the 4-fold virtual localization formula \eqref{virloc1} holds for some choice of orientation.
\end{thm}
\begin{proof}
Write $L_1$ for the first fibre direction of $X$ and $L_2$ for the second fibre direction. Then $L_1 \otimes L_2 \cong K_S$. Let $P_S := P_n(S,\beta)$ and denote the universal stable pair on $P_S \times S$ by $\mathbb{I}^{\mdot}_S = \{\oO_{P_S \times S} \rightarrow \mathbb{F}\}$. There exists a choice of orientation such that (e.g. \cite[Prop. 3.7]{CMT2}, \cite{CKM})
\begin{equation}[P_n(X,\beta)]^{\mathrm{vir}}=[P_n(S,\beta)]^{\mathrm{vir}}\cdot e\Big(-\dR\hH om_{\pi_{P_S}}(\mathbb{F}, \mathbb{F} \boxtimes L_1)\Big),
\nonumber \end{equation}
where the virtual class of $P_n(S,\beta)$ on RHS is defined using surface stable pair deformation-obstruction theory as in \cite[Appendix]{KT} and $\pi_{P_S} \colon P_S \times S  \to P_S$ denotes projection. The complex $-\dR\hH om_{\pi_{P_S}}(\mathbb{F}, \mathbb{F} \boxtimes L_1)=\EXT^1_{\pi_{P_S}}(\mathbb{F}, \mathbb{F} \boxtimes L_1)[-1]$ is concentrated in degree one.
Then
\begin{align*}
\int_{[P_n(X,\beta)]^{\vir}}\tau(\gamma)&=\int_{[P_n(S,\beta)]^{\vir}}\tau(\gamma)\cdot 
e\big(-\dR\hH om_{\pi_{P_S}}(\mathbb{F}, \mathbb{F} \boxtimes L_1)\big) \\
&=\int_{[P_n(S,\beta)]}\tau(\gamma)\cdot 
e\big(-\dR\hH om_{\pi_{P_S}}(\mathbb{F}, \mathbb{F} \boxtimes L_1)\big) \cdot e(\EXT^1_{\pi_{P_S}}(\mathbb{I}^\mdot_S,\mathbb{F})) \\
&=\sum_{[I^\mdot_S = \{\oO_S \rightarrow F\}]\in P_n(S,\beta)^{(\mathbb{C}^*)^4}}\tau(\gamma)|_{[I^\mdot_S]}\cdot 
e\big(-\RHom_{S}(F, F\boxtimes L_1)\big) \cdot e\big(-\RHom_{S}(I^{\mdot}_S, F)\big),
\end{align*}
where we use smoothness of the moduli space \cite[Prop.~5.6]{BF} in the second equality and Atiyah-Bott localization \cite{AB} for the action of $(\CC^*)^4$ on $P_S$ in the third equality.

Let $P_X := P_n(X,\beta)$, then $P_X \cong P_S$ by assumption, and therefore 
\begin{align*} \Ext^0_S(I^\mdot_S, F)\cong \Ext^1_X(I^\mdot_X, I^\mdot_X)_0, \end{align*}  
where $[I^\mdot_S = \{\oO_S \rightarrow F\}]\in P_n(S,\beta)$, $I_X^\mdot = \{ \oO_X \rightarrow \iota_* \oO_S \rightarrow \iota_* F\}$, and $\iota : S \hookrightarrow X$ denotes inclusion of the zero section. As shown in \cite[Prop. 3.7]{CMT2} and \cite{CKM}, at the level of obstruction spaces, we have an inclusion of a maximal isotropic subspace 
\begin{align*}
\Ext^1_S(I^\mdot_S, F)\oplus \Ext^1_S(F, F\otimes L_1)\hookrightarrow \Ext^2_X(I^\mdot_X, I^\mdot_X)_0
\end{align*}
with respect to the quadratic form given by Serre duality.
Hence 
\begin{align*}
&\sum_{[I^\mdot_S = \{\oO_S \rightarrow F\}]\in P_n(S,\beta)^{(\mathbb{C}^*)^4}}\tau(\gamma)|_{[I^\mdot_S]}\cdot 
e\big(-\RHom_{S}(F, F\boxtimes L_1)\big) \cdot e\big(-\RHom_{S}(I^{\mdot}_S, F)\big) \\
&= \sum_{[I^\mdot_X]\in P_n(X,\beta)^{(\mathbb{C}^*)^4}} (-1)^{o(\mathcal{L})|_{[I_X^\mdot]}} \cdot \tau(\gamma)|_{[I^\mdot_X]}\cdot 
\sqrt{(-1)^n\cdot e\big(\RHom_{X}(I^\mdot_X, I^\mdot_X)_0\big)}, 
\end{align*}
for appropriately chosen signs $(-1)^{o(\mathcal{L})|_{[I_X^\mdot]}}$.
\end{proof}
\begin{rmk} 
Although Theorem \ref{check loc formula} establishes the 4-fold localization formula \eqref{virloc1} when $P_n(X,\beta) \cong P_n(S,\beta)$ is smooth, it is not obvious that
the choice of signs on RHS of \eqref{virloc1} is compatible with the choice of signs in Conjecture \ref{affine DT/PT conj} and
Theorem \ref{affine implies toric}. 
\end{rmk}

\end{document}